\documentclass[10pt,oneside,english]{amsart}
\usepackage[utf8]{inputenc}
\usepackage[english]{babel}

\usepackage{lmodern}
\usepackage[T1]{fontenc}
\usepackage{textcomp}

\usepackage{hyperref}
\usepackage{fullpage}

\usepackage{amsmath,amssymb,amsthm}
\usepackage{mathrsfs}
\usepackage{mathtools}
\usepackage[all,cmtip]{xy}

\theoremstyle{plain}
\newtheorem{theorem}{Theorem}[section]
\newtheorem*{theorem*}{Theorem}
\newtheorem{proposition}[theorem]{Proposition}
\newtheorem{lemma}[theorem]{Lemma}
\newtheorem{corollary}[theorem]{Corollary}

\theoremstyle{definition}
\newtheorem{definition}[theorem]{Definition}

\theoremstyle{remark}
\newtheorem{remark}[theorem]{Remark}

\renewcommand\epsilon{\varepsilon}
\renewcommand\phi{\varphi}
\newcommand\RR{\mathbb{R}}
\newcommand\CC{\mathbb{C}}
\newcommand\ZZ{\mathbb{Z}}
\newcommand\QQ{\mathbb{Q}}

\newcommand\kk{\mathbf{k}}
\renewcommand\Im{\IIm}

\newcommand\Art{\mathbf{Art}}

\newcommand\Ohat{\widehat{\mathcal{O}}}
\newcommand\TW{\mathrm{TW}}

\DeclareMathOperator\id{id}
\DeclareMathOperator\Ker{Ker}

\DeclareMathOperator\IIm{Im}

\DeclareMathOperator*\End{End}
\DeclareMathOperator*\Hom{Hom}

\DeclareMathOperator\Ad{Ad}
\DeclareMathOperator\ad{ad}
\DeclareMathOperator\Gr{Gr}

\DeclareMathOperator\Def{Def}

\DeclareMathOperator\MHC{MHC}
\DeclareMathOperator\DR{DR}
\DeclareMathOperator\Res{Res}
\DeclareMathOperator\can{can}
\DeclareMathOperator\var{var}
\DeclareMathOperator\MF{MF}
\DeclareMathOperator\MFW{MFW}
\DeclareMathOperator\MH{MH}
\DeclareMathOperator\MHM{MHM}

\numberwithin{equation}{section}
\hypersetup{
pdfauthor = {Louis-Clément LEFÈVRE},
pdftitle = {Deformations of representations of fundamental groups of complex varieties},
pdfkeywords = {Representation Varieties, Fundamental Groups, Deformation Theory, Variations of Hodge Structure, Mixed Hodge Modules, L-infinity Algebras},
linktocpage = {true},
}

\title[Deformations of representations of fundamental groups]{Deformations of representations of fundamental groups of complex varieties}
\subjclass{14D07; 14C30; 18D50}
\keywords{Representation Varieties, Fundamental Groups, Deformation Theory, Variations of Hodge Structure, Mixed Hodge Modules, $L_\infty$-Algebras}

\author[L.-C. Lefèvre]{Louis-Clément Lefèvre}

\email{Louis-Clement.Lefevre@math.cnrs.fr}
\address{Universität Duisburg-Essen, Fakultät Für Mathematik, 45117 Essen, Germany}
\curraddr{Lycée Hoche, 73 avenue de Saint-Cloud, 78000 Versailles, France}

\begin{document}

\begin{abstract}
We describe locally the representation varieties of fundamental groups for smooth complex varieties at representations coming from the monodromy of a variation of mixed Hodge structure. Given such a manifold $X$ and such a linear representation $\rho$ of its fundamental group $\pi_1(X,x)$, we use the theory of Goldman-Millson and pursue our previous work that combines mixed Hodge theory with derived deformation theory to construct a mixed Hodge structure on the formal local ring $\widehat{\mathcal{O}}_\rho$ to the representation variety of $\pi_1(X,x)$ at $\rho$. Then we show how a weighted-homogeneous presentation of $\widehat{\mathcal{O}}_\rho$ is induced directly from a splitting of the weight filtration of its mixed Hodge structure. In this way we recover and generalize theorems of Eyssidieux-Simpson ($X$ compact) and of Kapovich-Millson ($\rho$ finite).
\end{abstract}


\maketitle

\tableofcontents

\section{Introduction}

\subsection{Topology of algebraic varieties}
Our object of study will be either smooth complex algebraic varieties, that we can study by embedding them into a smooth proper algebraic variety $\overline{X}$, or complex manifolds admitting a compactification into a compact Kähler manifold $\overline{X}$ (these are called \emph{quasi-Kähler}, in analogy with quasi-projective varieties). In both cases the cohomology groups of $\overline{X}$ carry pure Hodge structures; the complement of $X\subset\overline{X}$ can be taken to be a divisor with normal crossings.

Then the cohomology groups of such an $X$ come equipped with \emph{mixed Hodge structures} (MHS), defined and constructed first by Deligne \cite{DeligneII}. In particular they have a \emph{weight filtration} defined over $\QQ$ whose graded quotients behave as cohomology groups of compact Kähler manifolds. Then mixed Hodge structures have been constructed on other topological invariants of such varieties, in particular on the rational homotopy groups by Morgan \cite{Morgan} and Hain \cite{Hain}. They can be used to give concrete restrictions on the possible fundamental groups of these varieties \emph{via} the study of weights.

In previous work \cite{Lefevre2} we constructed MHS's on certain invariants associated to linear representations of the fundamental group $\pi_1(X,x)$. We want to pursue this work and analyze the restrictions that this gives on the representations.

\subsection{Deformations of representations of fundamental groups}
For $X$ as above, $\pi_1(X,x)$ is always a finitely presentable group. Hence the set of representations of $\pi_1(X,x)$ into some fixed linear algebraic group $G$ over $\CC$ has a natural structure of affine scheme that we denote by $\Hom(\pi_1(X,x),G)$. Given such a representation $\rho$, our main object of study is the formal completion of the local ring to $\Hom(\pi_1(X,x),G)$ at its point $\rho$. We denote it by $\Ohat_\rho$. Its associated deformation functor, a functor from the category $\Art$ of local Artin $\CC$-algebras with residue field $\CC$ to the category of sets, is given by
\begin{equation}
A\in \Art \longmapsto \Def_{\Ohat_\rho}(A) := \Hom(\Ohat_\rho, A)
\end{equation}
and is canonically isomorphic to the functor of formal deformations of $\rho$
\begin{equation}
\Def_\rho(A):= \bigl\{ \tilde{\rho}:\pi_1(X,x)\rightarrow G(A) \bigm|
\tilde{\rho}=\rho\ \text{over}\ G(\CC) \bigr\} .
\end{equation}
Thus, if one can give a concrete presentation of $\Ohat_\rho$ as
\begin{equation}
\label{equation:presentation-Orho}
\Ohat_\rho \simeq \CC[[X_1,\dots,X_n]]/(P_1,\dots,P_r)
\end{equation}
where $P_1,\dots, P_r$ are polynomials then one can describe completely the deformation theory of $\rho$: deformations over $A$ are given by elements $x_1,\dots,x_n$ of $A$ satisfying the relations $P_1(x_1,\dots,x_n)=\dots=P_r(x_1,\dots,x_n)=0$. In particular, if an integer $N$ is bigger than the maximal of the degrees of $P_1,\dots,P_r$ then there are no obstructions to lifting deformations of order $N$ to deformations of order $N+1$.

This was described first by Goldman and Millson in \cite{GoldmanMillson} in the case where $X$ is a compact Kähler manifold and $\rho$ is the monodromy of a polarizable variation of Hodge structure over $X$: they show that $\Ohat_\rho$ has a quadratic presentation, that is, a presentation with $P_1,\dots,P_r$ of degree at most $2$.

A presentation of $\Ohat_\rho$ was also obtained by Kapovich-Millson \cite{KapovichMillson} when $X$ may be non-compact and $\rho$ is a representation with finite image: they show that there is a presentation with weights $1,2$ on the variables $X_1,\dots,X_n$ and polynomials $P_1,\dots,P_r$ that are homogeneous of possible degrees $2,3,4$ with respects to these weights on the variables.

These two results have concrete applications to exhibit classes of groups that cannot be fundamental groups of smooth varieties. Our motivation is to generalize these results and unify them, and make more explicit the use of mixed Hodge structures.

\begin{theorem}
Let $\rho$ be the monodromy representation of an admissible variation of mixed Hodge structure over a smooth complex algebraic variety or a quasi-Kähler manifold $X$ with unipotent monodromy at infinity. Then there is a mixed Hodge structure on $\Ohat_\rho$, functorial in $X$ and the base point $x$. Furthermore, one obtains a weighted-homogeneous presentation of $\Ohat_\rho$ by splitting the weight filtration.
\end{theorem}

In our previous work, we constructed this MHS on $\Ohat_\rho$ only in the case treated by Goldman-Millson (recovering a result of Eyssidieux-Simpson \cite{EyssidieuxSimpson}) and in the case of finite monodromy treated by Kapovich-Millson. This already provided some unification of these two theories that was suitable for further generalizations, enlightening the role of mixed Hodge theory, but we were not able to fully recover the weighted-homogeneous presentation of Kapovich-Millson. It will now appear clearly that the weights $1,2$ on generators come from the MHS of $H^1(X,\QQ)$ and the weights $2,3,4$ on the relations come from the MHS of $H^2(X,\QQ)$.

\subsection{Deformations via DG Lie algebras}
Let $\rho$ be a representation of $\pi_1(X,x)$ into a linear algebraic group $G$ with Lie algebra $\mathfrak{g}$, as above, at a fixed base point $x$. To this is associated a flat principal bundle $P$ with associated bundle $\ad_\rho:=\ad(P)=P\times^G \mathfrak{g}$: this is a local system of Lie algebras over $X$. Let $L$ be the algebra of $\mathcal{C}^\infty$ differential forms over $X$ with values in $\ad_\rho$. This has a natural structure of differential graded (DG) Lie algebra, with an augmentation $\epsilon$ to $\mathfrak{g}$ that evaluates forms of degree zero at $x$.

The theory introduced by Goldman-Millson states that it is possible to associate to $L$ a deformation functor $\Def_L$, and actually to $(L,\epsilon)$ a small modification $\Def_{L,\epsilon}$, again from $\Art$ to sets, such that $\Def_{L,\epsilon}$ is isomorphic to $\Def_\rho$ and such that a quasi-isomorphism of DG Lie algebras $L\rightarrow M$ induces an isomorphism of deformation functors $\Def_{L,\epsilon} \rightarrow \Def_{M,\epsilon}$. So the local ring $\Ohat_\rho$ is determined by the data of $L$ up to quasi-isomorphism, and its obstruction theory (i.e.\ the presentation~\eqref{equation:presentation-Orho}) can be understood in terms of $L$ only (and of $\epsilon$).

The conclusion in the compact cases that they treat follows quite easily once proven that, using the particular properties of harmonic analysis on Kähler manifolds, $L$ is \emph{formal}, that is, quasi-isomorphic as DG Lie algebra to its cohomology algebra $H(L)$, which is finite-dimensional and with zero differential. Then $\Ohat_\rho$ is essentially isomorphic to the formal local ring at $0$ to the equation $[\omega,\omega]=0$ for $\omega$ in $H^1(L)$ and with values in $H^2(L)$, and this is a quadratic presentation.

In order to pursue this method to the non-formal case, we argued in our previous work  \cite[\S~4]{Lefevre2} that the right tool to use is $L_\infty$ algebras. Namely on $H(L)$ there exists a sequence of multilinear \emph{higher operations} (for each $n\geq 1$, a multilinear operation in $n$ variables on $H(L)$, which for $n=2$ is the Lie bracket, for $n=1$ the differential which is zero here) arranged in a certain algebraic and combinatorial structure called \emph{$L_\infty$ algebra} such that $L$ and $H(L)$ become quasi-isomorphic \emph{as $L_\infty$ algebras} and such that $\Def_L$ can be written in terms of $H(L)$ with its $L_\infty$ structure only. So, we can again understand $\Def_L$ via linear algebra in finite-dimensional vector spaces but at the cost of working with more operations. A similar argumentation is central in a recent work of Budur-Rubi\'o \cite{BudurRubio}.

\subsection{Hodge theory}
\label{subsection:Hodge}
In several special cases, the cohomology of $X$ with coefficients in a local system $\mathbb{V}$ also carries a MHS. If $\rho$ is the monodromy representation of $\mathbb{V}$ then the associated DG Lie algebra $L$ as above also carries a MHS on cohomology and the Lie bracket is a morphism of MHS. What we want to exploit is precisely the interaction between the MHS on $H^\bullet(L)$ and the deformation functor of $L$.

A very interesting class of such local systems is provided by the \emph{variations of (mixed) Hodge structure} (VHS, VMHS) (\S~\ref{section:VMHS}). These arise naturally from geometry as follows. Given a morphism $f : Y \rightarrow X$ of algebraic varieties, then each fiber $f^{-1}(x) \subset Y$ is an algebraic variety and carries a MHS on its cohomology. If $f$ is smooth and proper, the assignment
\begin{equation*}
x\in X\longmapsto  H^i(f^{-1}(x), \mathbb{Q}) = (R^i f_* \underline{\mathbb{Q}}_Y)_x
\end{equation*}
forms the underlying local system of $\QQ$-vector spaces of a VHS ; it is more generally a VMHS if $f$ is a stratified fiber bundle. The fact that for a VMHS $\mathbb{V}$ then $H^\bullet(X, \mathbb{V})$ also carries a MHS was conjectured by Deligne and proved first in the case where $X$ is compact. This then motivated a long development of the theory of VHS's and VMHS's and their behavior around a normal crossing divisor at infinity $D := \overline{X} \setminus X$. This MHS was constructed in the following cases:

\begin{enumerate}
\item When $X$ is compact Kähler, due to Deligne and written by Zucker \cite[\S~1--2]{Zucker} for the case of VHS.
\item When $\mathbb{V}$ is a VHS over $X$ that extends to $\overline{X}$, this is easily deduced.
After a theorem of Griffiths, this is the same as requiring that $\mathbb{V}$ as a local system extends to $\overline{X}$, equivalently that its monodromy is trivial around $D$.
\item When $X$ is one-dimensional and $\mathbb{V}$ is a VHS, due to Zucker \cite[\S~13]{Zucker}.
\item When $X$ is one-dimensional and $\mathbb{V}$ is a VMHS, due to Steenbrink-Zucker \cite{SteenbrinkZucker}. The VMHS has to satisfy the \emph{admissibility condition}, that they introduce, relatively to $\overline{X}$; in particular the monodromy at infinity has to be \emph{quasi-unipotent}.
\item In the most general case, when $\mathbb{V}$ is a VMHS on $X$ admissible relatively to $\overline{X}$, due to Saito \cite{SaitoMHM} using his theory of \emph{mixed Hodge modules}.
\end{enumerate}

The MHS on the cohomology of $L$ actually always comes from a pre-existing structure at the level of $L$ called \emph{mixed Hodge complex}. This is the notion introduced by Deligne in \cite{DeligneIII}, consisting of two complexes: $L_\QQ$ over $\QQ$ carrying the weight filtration, and $L_\CC$ over $\CC$ carrying the Hodge and weight filtrations, with a quasi-isomorphism $L_\QQ\otimes\CC \approx L_\CC$ and axioms ensuring that the weight-graded pieces of cohomology glue to pure Hodge structures, defining a MHS on each  $H^\bullet(L)$.

In our previous work \cite[\S~7--9]{Lefevre2} we explained that the only input of geometry that we need in order to construct a MHS on $\Ohat_\rho$ is to have such an $L$ which is at the same time a DG Lie algebra and a mixed Hodge complex. This is not so obvious because the naive way of constructing $L_\QQ$, via sheaf theory or via singular cochains, does not provide a DG Lie algebra: the cup-product is not commutative at the level of cochains. The global picture is quite similar to the work of Morgan where in order to construct the MHS on rational homotopy groups, it is necessary to have an object $A$ which is at the same time a mixed Hodge complex (thus containing all the information on the MHS on $H^\bullet(X)$) and a commutative DG algebra (by rational homotopy theory, this up to quasi-isomorphism contains all the information about the rational homotopy type of $X$). So he constructs $A_\QQ$ which is an algebra of rational polynomial differential forms on~$X$ with logarithmic poles along~$D$. Similarly, in our work $L_\QQ$ is an algebra of rational polynomial differential forms with coefficients in a rational local system. Fortunately, the Thom-Whitney functors, introduced by Navarro Aznar \cite{Navarro} when reviewing this work of Morgan, allow us to give a straightforward and functorial construction of this~$L_\QQ$.

The data of such a structure on $L$ is usually not unique (it depends on the chosen compactification $\overline{X}$) but it is unique up to quasi-isomorphism. The construction of our MHS on $\Ohat_\rho$ from $L$ is directly seen to be functorial and invariant under quasi-isomorphisms (of simultaneous structures of DG Lie algebra and mixed Hodge complexes) of $L$. So, combining, the MHS on $\Ohat_\rho$ is functorial in $X$ and the base point~$x$:
a morphism of varieties with given base points and compactifications
\begin{equation*}
f : (X, \overline{X}, x) \longrightarrow (Y, \overline{Y}, y)
\end{equation*}
induces a pull-back morphism of representations
\begin{equation*}
f^* : \Hom(\pi_1(Y,y), G) \longrightarrow \Hom(\pi_1(X,x), G)
\end{equation*}
which also pulls back VMHS's and pulls back our whole construction of $L$, and then a morphism of local rings
\begin{equation*}
\Ohat_\rho \rightarrow \Ohat_{f^*\rho}
\end{equation*}
and this is a morphism of MHS. Usual arguments imply that the MHS is actually independent of the compactification.

\subsection{Results}
\label{subsection:results}
Our method divides the construction of a MHS on $\Ohat_\rho$ into two well-distinct parts. In the first part, we construct $L$ which is at the same time a DG Lie algebra and a mixed Hodge complex. This strongly depends on the ``geometric situation'' that we fix (the given hypotheses for $X$ and $\rho$), hence through \S~\ref{section:non-degenerate}--\ref{section:MHM} we will devote one section to each of the cases listed in \S~\ref{subsection:Hodge}. The second part is purely algebraic and refers to $L$ only and not any more to geometry. We treat this only in \S~\ref{section:applications}.

For the first part, our result can be summed up as follows.

\begin{theorem}
Let $\rho$ be the monodromy representation of an admissible (graded-polarizable) variation of mixed Hodge structure over $X\subset \overline{X}$, with unipotent monodromy at infinity. Then one can construct an object $L$ which is at the same time a DG Lie algebra and a mixed Hodge complex which computes the cohomology of $X$ with coefficients in the local system $\ad_\rho$. It is functorial in $(X,\overline{X},x)$.
\end{theorem}

By the main result of \cite{Lefevre2} this is enough to construct the MHS on $\Ohat_\rho$, functorial in $X,x$ only.

To put our work in a slightly more general context that costs us nothing more,
we actually construct a mixed Hodge complex $\MHC(X,\overline{X},\mathbb{V})$ attached to the admissible VMHS $\mathbb{V}$ with unipotent monodromy at infinity, functorial in $(X,\overline{X})$ and in $\mathbb{V}$.
The tensor product of two such VMHS $\mathbb{V}_1,\mathbb{V}_2$ is again a VMHS and our MHC will be equipped with canonical maps
\begin{equation*}
\MHC(X,\overline{X},\mathbb{V}_1) \otimes \MHC(X,\overline{X},\mathbb{V}_2)
\longrightarrow \MHC(X,\overline{X},\mathbb{V}_1 \otimes \mathbb{V}_2)
\end{equation*}
(together with an identity and associativity condition)
that are compatible with the interchange map of $\mathbb{V}_1$ and $\mathbb{V}_2$. In particular if $\mathbb{V}$ is a VMHS then so is $\ad_\rho$ (if $G\subset GL(\mathbb{V}_x)$ then $\ad_\rho\subset \End(\mathbb{V})$) equipped with the Lie bracket $[-,-]$ and then there is an induced Lie bracket on $\MHC(X,\overline{X},\ad_\rho)$ via the composition
\begin{equation*}
\MHC(X,\overline{X}, \ad_\rho) \otimes \MHC(X,\overline{X},\ad_\rho) \\
\longrightarrow \MHC(X,\overline{X},\ad_\rho \otimes \ad_\rho)
\xrightarrow{[-,-]} \MHC(X,\overline{X},\ad_\rho) .
\end{equation*}
Typically $\MHC(X,\overline{X},\mathbb{V})$ is called a \emph{lax symmetric monoidal functor} (\S~\ref{section:lax-monoidal}) in the data of $\mathbb{V}$, and by this same argument inherits every kind of algebraic structure that can be described on $\mathbb{V}$ by multilinear maps that eventually exchange the factors. This generalization matters because in future work \cite{Lefevre4} we will study the structure of Lie module of $\mathbb{V}$ over $\ad_\rho$, also in the sense of VMHS, and we will immediately get also a mixed Hodge complex $\MHC(X,\overline{X},\mathbb{V})$ that is a module over $\MHC(X,\overline{X},\ad_\rho)$ with action map
\begin{equation*}
\MHC(X,\overline{X}, \ad_\rho) \otimes \MHC(X,\overline{X},\mathbb{V}) \\
\longrightarrow \MHC(X,\overline{X},{\ad_\rho} \otimes {\mathbb{V}})
\longrightarrow \MHC(X,\overline{X},\mathbb{V}) .
\end{equation*}

Now, we study the interaction between the deformation functor of $L$, that can be written in $H(L)$ with higher operations of $L_\infty$ algebras, and the existing MHS on $H(L)$. What we show in \S~\ref{section:applications} is that there is a splitting of the weight filtration on $H(L)$ that is compatible with the higher operations of $L_\infty$ algebra. Because the higher operations are multilinear and respect the weights, but the weights of $H(L)$ in each degree are limited, this forces many higher operations to vanish and the remaining ones give the equations for the deformation functor. The most general result, in its purely algebraic version, is:

\begin{theorem}
Assume that the MHS on $H^1(L)$ has only strictly positive weights. Then $\Ohat_\rho$ has a weighted-homogeneous presentation as in~\eqref{equation:presentation-Orho} with weights of generators coming from weights of $H^1(L)$ and weights of relations limited to the weights of $H^2(L)$.
\end{theorem}

This relates as follows to the case treated by Kapovich-Millson: the MHS on $H^1(L)$ directly comes from the MHS on $H^1(X)$ and has only weights $1,2$, and the MHS on $H^2(L)$ directly comes from the one of $H^2(X)$ and has only weights $2,3,4$. Our method allows us to perform such a detailed analysis from the weights of $H^\bullet(L)$. We sum this up in our last section~\ref{subsection:consequences}.

In particular the hypothesis on the weights of $H^1(L)$ is satisfied when $\rho$ comes from a VHS because then $\ad_\rho$ is a VHS of weight zero. So we get:

\begin{theorem}
If $\rho$ is the monodromy representation of an admissible VHS over $X$ with unipotent monodromy at infinity then $\Ohat_\rho$ has a weighted-homogeneous presentation.
\end{theorem}

The precise understanding of the weights depends on the behavior of the VHS near infinity and this is easier so see in the case of curves. So once again we find interesting to treat these cases separately. For VMHS's, our method does not allow us to show that the presentation is finite, but it sill allows to write down weighted-homogeneous equations.

\subsection{Acknowledgments}
I thank notably Yohan Brunebarbe, Nero Budur, Joana Cirici, Brad Drew, Clément Dupont, Konstantin Jakob, Marcel Rubi\'o,
for useful conversations all along the preparation of this work. I thank particularly Philippe Eyssidieux and Jochen
Heinloth for conversations as well as for comments on the final version of this work.


\subsection{Some updates}
Several years have past since this prepublication was first uploaded and sent for review; in the meantime I left Academia for teaching. The claim that we could deal with representations with local quasi-unipotent monodromy has been removed and reduced to only local unipotent monodromy. The important article~\cite{BakkerBrunebardeTsimerman} from Bakker-Brunebarbe-Tsimerman appeared and vastly generalizes our results in the language of mixed twistor structures (\S~4 therein) with applications to the Shafarevich conjecture on universal covers of complex algebraic varieties in the non-proper case. We thank the referee for his valuable comments on the first submission of this work an apologize for the long time it took to fix it. 

\section{Reminder on mixed Hodge theory}
\label{section:Hodge}

We will denote by $\kk$ a fixed sub-field of $\RR$. We always denote increasing filtrations by a lower index and decreasing filtrations by an upper index: $W_\bullet, F^\bullet$. Filtrations of cochain complexes are always assumed to be biregular, i.e.\ they restrict to a finite filtration in each degree.

\subsection{Mixed Hodge structures}
Recall briefly that a \emph{pure Hodge structure} (HS) of weight $k$ over $\kk$ is given by a finite-dimensional vector space $H_\kk$ over $\kk$ with a decomposition $H_\CC:=H_\kk\otimes\CC=\bigoplus_{p+q=k} H^{p,q}$ with $\overline{H^{p,q}}=H^{q,p}$, complex conjugation being taken with respect to $H_\kk\subset H_\CC$. In this case the Hodge filtration $F^\bullet$ is defined as $F^p H_\CC:= \bigoplus_{r\geq p} H^{r,q}$.

A \emph{mixed Hodge structure} (MHS) is given by a finite-dimensional vector space $H_\kk$ over $\kk$ equipped with an increasing filtration $W_\bullet$ and a decreasing filtration $F^\bullet$ of $H_\CC$, such that each term $\Gr^W_k H_\kk:=W_k H_\kk/W_{k-1}H_\kk$ with the induced filtration $F^\bullet$ over $\CC$ is a pure Hodge structures of weight $k$, whose bigraded terms are again denoted by $H^{p,q}$.

For a pure Hodge structure $H$ of weight $k$, a \emph{polarization} is given by a bilinear form $Q:H_\kk\otimes H_\kk \rightarrow \kk$ which is $(-1)^k$-symmetric and satisfies the two relations
\begin{enumerate}
\item $Q(H^{p,q}, H^{r,s})=0$ if $(p,q)\neq (r,s)$,
\item $i^{p-q} Q(v,\bar{v})>0$ for $0\neq v\in H^{p,q}$.
\end{enumerate}

Mixed Hodge structures form an abelian category with tensor product and internal Hom. Hence one can consider various kinds of algebras internally to MHS: for example a Lie bracket $[-,-]$ on the MHS $H$ has to be defined over $\kk$ and satisfy $[W_k H_\kk, W_\ell H_\kk]\subset W_{k+\ell} H_\kk$ and $[F^p H_\CC, F^r H_\CC]\subset F^{p+r} H_\CC$.

A direct sum of pure HS's of various weights is automatically a MHS with a split weight filtration. Such MHS's are said to be \emph{split}. In general however, MHS's coming from the cohomology of complex varieties are not split. Nevertheless, there exists a canonical way of splitting a MHS $H$ over $\CC$ by defining subspaces $A^{p,q}\subset H_\CC$ that projects to $H^{p,q}$. This splitting is functorial and compatible with tensor products. In this work, after having constructed a MHS, we will be mainly interested in describing such a splitting.

\subsection{Variations of mixed Hodge structure}
\label{section:VMHS}
Let $X$ be any complex manifold. We will always use the letter $\Omega_X^\bullet$ to denote the complex of holomorphic differential forms on~$X$.

\begin{definition}
A \emph{variation of Hodge structure} (VHS) of weight $k$ over $X$ is the data of a local system $\mathbb{V}$ of finite-dimensional $\kk$-vector spaces, to which is associated a flat connection $\nabla$ on the holomorphic vector bundle $\mathcal{V}:=\mathbb{V}\otimes_{\kk}\mathcal{O}_X$,
and a filtration of $\mathcal{V}$ by holomorphic sub-vector bundles $\mathcal{F}^\bullet$, such that at each point $x\in X$ the data $(\mathcal{V}_x, \mathcal{F}^\bullet_x)$ forms a pure Hodge structure of weight $k$, and furthermore $\nabla$ satisfies the transversality condition $\nabla(\mathcal{F}^p)\subset \Omega_X^1 \otimes\mathcal{F}^{p-1}$.

A \emph{polarization} for the VHS $\mathbb{V}$ is a flat bilinear pairing defined over $\kk$, $Q:\mathbb{V}\otimes\mathbb{V}\rightarrow \kk$, such that each $(\mathcal{V}_x, \mathcal{F}^\bullet_x, Q_x)$ is a polarized Hodge structure.
\end{definition}

In the following we will work only with polarizable variations of Hodge structure.

\begin{definition}[{\cite[3.4]{SteenbrinkZucker}}]
A \emph{variation of mixed Hodge structure} (VMHS) over $X$ is the data of a local system $\mathbb{V}$ of finite-dimensional vector spaces over $\kk$ and a filtration $W_\bullet$ of $\mathbb{V}$ by sub-local systems $\mathbb{W}_\bullet$, together with a filtration as above of $\mathcal{V}=\mathbb{V}\otimes_{\kk}\mathcal{O}_X$ by sub-vector bundles $\mathcal{F}^\bullet$, such that the induced flat connection again satisfies $\nabla(\mathcal{F}^p)\subset \Omega_X^1 \otimes\mathcal{F}^{p-1}$ and furthermore each $\Gr^{W}_k\mathbb{V}$ with the induced $\mathcal{F}^\bullet$ forms a variation of Hodge structure of weight~$k$.

It is said to be \emph{graded-polarizable} if each term $\Gr^W_k\mathbb{V}$ is a polarizable VHS.
\end{definition}

\subsection{Mixed Hodge complexes}
In the original article of Deligne, mixed Hodge structures on the cohomology groups of $X$ are constructed by first compactifying $X$ into $j:X\hookrightarrow \overline{X}$ with complement a divisor $D$ with normal crossings and defining the weight and Hodge filtrations on the complex of forms with logarithmic poles $\Omega_X^\bullet(\log D)$, and a weight filtration on the complex $Rj_{*} \underline{\QQ}_X$. Then showing that these two complexes are related via a $W$-filtered quasi-isomorphism, with the $W$-spectal sequence computing the weight-graded pieces of $H^\bullet(X,\QQ)$. This last step has been abstracted into the notion of mixed Hodge complex: these are filtered complexes, naturally living in a filtered derived category, with axioms to check that will ensure that their cohomology carries a MHS in each degree.

\begin{definition}
A \emph{pure Hodge complex of weight $k$} is the data of a bounded-below complex $K^\bullet_\kk$ over $\kk$ with each $H^i(K_\kk)$ finite-dimensional, a bounded-below filtered complex $(K^\bullet_\CC, F^\bullet)$ over $\CC$, and a chain of quasi-isomorphisms $K_\kk\otimes\CC \approx K_\CC$, such that
\begin{enumerate}
\item The differential $d$ of $K_\CC$ is strictly compatible with $F$,
\item The filtration $F$ defines a pure Hodge structure of weight $i+k$ on $H^i(K_\CC)$, with a form over $\kk$ coming
from $H^i(H_\kk)$.
\end{enumerate}

A \emph{mixed Hodge complex} (MHC) is the data of a filtered bounded below-complex $(K^\bullet_\kk,W_\bullet)$ over $\kk$ with each $H^i(K_\kk)$ finite-dimensional, a bifiltered bounded-below complex $(K^\bullet_\CC, W_\bullet, F^\bullet)$ over $\CC$, a filtered quasi-isomorphism $(K_\kk, W)\otimes\CC \approx (K_\CC,W)$, such for all $k$ the data of $\Gr^W_k(K_\kk)$ and $(\Gr^W_k(K_\CC),F^\bullet)$ forms a pure Hodge complex of weight~$k$.
\end{definition}

The main theorem of Deligne is that for a mixed Hodge complex $K$ the $W$-spectral sequence of $K$ degenerates at $E_2$ and forms the weight-graded pieces of a MHS on each $H^i(K)$ with the induced filtration $F$ and the shifted filtration $W[i]$ with $W[i]_k:=W_{i+k}$.

For the needs of rational homotopy theory, Morgan introduced the notion of \emph{mixed Hodge diagram} which is simply a mixed Hodge complex $K$ whose both components $K^\bullet_\kk$, $K^\bullet_\CC$ are commutative DG algebras (non-negatively graded, and with multiplication compatible with the filtrations) and the quasi-isomorphisms relating them are quasi-isomorphisms of commutative DG algebras. Though being algebraically a simple variant of the notion of mixed Hodge complexes, the construction of such mixed Hodge diagrams is more complicated because the commutative DG algebra $K^\bullet_\kk$ does not come from usual differential forms nor from singular cochains, see \S~\ref{section:ThomWhitney}. Similarly we introduced:

\begin{definition}[{\cite[\S~8]{Lefevre2}}]
A \emph{mixed Hodge diagram of Lie algebras} is a mixed Hodge complex $L$ whose both components $L_\kk$, $L_\CC$ are also DG Lie algebras, with bracket compatible with the filtrations, and the quasi-isomorphisms between them being quasi-isomorphisms of DG Lie algebras.
\end{definition}

Mixed Hodge complexes can be naturally considered into their derived category, where a quasi-isomorphism between the MHC $K,L$ is given by filtered quasi-isomorphisms between $(K_\kk,W) \approx (L_\kk, W)$ and bifiltered quasi-isomorphisms $(K_\CC,W,F) \approx (L_\CC, W,F)$. The notions of pure and mixed Hodge complex also exist at the level of sheaves, and this will be more practical for us. Let us assume that we work with $R\Gamma$ defined using the Godement resolution for sheaves, that also provides us with functorial resolution as well as filtered and bifiltered resolutions.

\begin{definition}
Let $X$ be a topological space. A \emph{pure Hodge complex of sheaves} of weight $k$ over $X$ is the data of a bounded-below complex of sheaves $\mathcal{K}^\bullet_\kk$ over $\kk$, a bounded-below filtered complex of sheaves $(\mathcal{K}^\bullet_\CC,F^\bullet)$ over $\CC$, and a quasi-isomorphism $\mathcal{K}_\kk\otimes\CC \approx \mathcal{K}_\CC$, such that applying $R\Gamma(X,-)$ to all this data (i.e.\ to the complexes, the filtrations, and the quasi-isomorphisms between them) gives a pure Hodge complex of weight~$k$.

Similarly, a \emph{mixed Hodge complex of sheaves} over $X$ is the data of a filtered bounded-below complex of sheaves $(\mathcal{K}^\bullet_\kk,W_\bullet)$ over $\kk$, a bifiltered bounded-below complex of sheaves $(\mathcal{K}^\bullet_\CC,W_\bullet,F^\bullet)$ over $\CC$, and a filtered quasi-isomorphism $(\mathcal{K}_\kk,W)\otimes\CC \approx (\mathcal{K}_\CC,W)$ such that applying $\Gr^W_k$ to all this data gives a pure Hodge complex of sheaves of weight $k$; in this case applying $R\Gamma(X,-)$ gives a mixed Hodge complex.
\end{definition}

\begin{remark}
\label{remark:C-MHS}
Mixed Hodge structures and mixed Hodge complexes can also be defined over the base field $\kk=\CC$ but this requires an appropriate modification of the definitions, see for example \cite[\S~1]{EyssidieuxSimpson}. Essentially we drop the complex conjugation but we introduce a third filtration, denoted by $\overline{G}^\bullet$, playing the role of the conjugate $\overline{F}$ of $F$. The properties of $\overline{F}$ that would have followed simply by conjugation from $F$ need to be included as axioms for $\overline{G}$. In any MHC over $\kk$, keeping only its component over $\CC$ and the three filtrations $W,F,\overline{F}$ defines a MHC over $\CC$.
\end{remark}

\section{Thom-Whitney functors}
\label{section:ThomWhitney}
\subsection{Lax monoidal symmetric functors}
\label{section:lax-monoidal}
Our results are actually easier to express using a bit more of categorical notions around the tensor product. Recall briefly (see \cite[VII, XI]{MacLane} and \cite[B.3]{LodayVallette}) that a \emph{monoidal category} is a category $\mathcal{C}$ equiped with a tensor product $\otimes$ and an object $1$, together with natural associativity isomorphsims $x\otimes (y\otimes z) \simeq (x\otimes y)\otimes z$ and $1\otimes x\simeq x$, $x\otimes 1 \simeq x$; these isomorphisms themselves have to satisfy further coherence axioms such as the commutativity of the pentagonal diagram. The monoidal category $\mathcal{C}$ is said to be \emph{symmetric} if for any two objects $x,y$ there is an isomorphism $s_{x,y}:x\otimes y \simeq y \otimes x$ and these isomorphisms furthermore satisfy an additionnal coherence axiom, the hexagonal diagram.

The following categories are well-known to be symmetric monoidal:
\begin{enumerate}
\item vector spaces, or sheaves of vector spaces,
\item MHS, VMHS, admissible VMHS with unipotent monodromy at infinity ; the tensor product does, on one side, the tensor product of the underlying local systems (and of their filtrations), and on the other side, the tensor product of vector bundles with connection,
\item \label{item:symmetric-monoidal-cochain-complex}
cochain complexes, sheaves of cochain complexes; in this case the isomorphism $s_{K,L}:K\otimes L \simeq L\otimes K$ sends an element $x\otimes y$ (with $x,y$ of respective degrees $i,j$) to $(-1)^{ij} y\otimes x$,
\item MHC, MHC of sheaves.
\end{enumerate}

In any such category, one can define a \emph{commutative algebra object} to be an object $x$ with a multiplication $\mu:x\otimes x\rightarrow x$ and a unity $1\rightarrow x$, satisfying natural associativity, unity and commutativity axioms. In particular the commutativity condition can be written as $\mu\circ s_{x,x} = \mu$. In the case of cochain complexes (\ref{item:symmetric-monoidal-cochain-complex}) above, it is the particular choice of the isomorphism $s_{x,y}$ that is responsible for the Koszul rule of sign $\mu(y,x)=(-1)^{ij}\mu(x,y)$.

One can also define a \emph{Lie algebra object} in $\mathcal{C}$ as an object $L$ with a bracket $B: L\otimes L\rightarrow L$ that is anti-symmetric (here this means $B\circ s_{L,L}=-B$) and satisfies the Jacobi identity (this can be written with compositions of $B$ and of permutations of $L\otimes L\otimes L$). So for example, a mixed Hodge diagram of Lie algebras is just a Lie algebra object in the symmetric monoidal category of MHC.

There is a natural notion of functor between such monoidal categories that we will encounter. For example, let us recall that in the Theorem of Kunneth for cochain complexes $K,L$, there is first a natural map
\begin{equation*}
H(K)\otimes H(L) \longrightarrow H(K\otimes L)
\end{equation*}
that is, when working over a field, an isomorphism. In particular, if $K$ alone has a structure of associative algebra, then $H(K)$ becomes an associative algebra by the composition
\begin{equation*}
H(K)\otimes H(K) \longrightarrow H(K\otimes K) \longrightarrow H(K)
\end{equation*}
where the last map is induced by the multiplication of $K$. Here is the categorical notion:

\begin{definition}
\label{definition:lax-monoidal}
A \emph{lax monoidal functor} between the monoidal categories $(\mathcal{C},\otimes, 1_{\mathcal{C}})$ and $(\mathcal{D}, \otimes, 1_{\mathcal{D}})$ is the data of a functor beween categories $T:\mathcal{C}\rightarrow\mathcal{D}$ together with natural maps for all objects $x, y$ of $\mathcal{C}$
\begin{equation*}
T(x)\otimes T(y)\longrightarrow T(x\otimes y)
\end{equation*}
and a map
\begin{equation*}
1_{\mathcal{D}} \longrightarrow T(1_{\mathcal{C}})
\end{equation*}
satisfying further coherence axioms (\cite[XI.2]{MacLane}, \cite[B.3.3]{LodayVallette} to relate them with the associativity and unity of $\mathcal{C}$ and $\mathcal{D}$.
\end{definition}

These are simply called \emph{monoidal functors} in \cite[XI.2]{MacLane}, that uses the term of \emph{strong} functor when the above maps are actually isomorphisms and \emph{strict} functor when they are identities. So the lax condition is a weakening of these.

Now let us come to our main point by recalling the construction of the cup-product using sheaf cohomology with coefficients in $\QQ$ and the Godement resolution for sheaves, over a topological space $X$. The sheaf $\underline{\QQ}_X$ has a resolution $\underline{\QQ}_X\rightarrow \mathscr{F}^\bullet$ and the cup-product is an associative bilinear operation $\mathscr{F}\otimes\mathscr{F}\rightarrow\mathscr{F}$ that is not yet commutative. Then we take global sections over $X$ to define a map
\begin{equation*}
\Gamma(X,\mathscr{F})\otimes\Gamma(X,\mathscr{F})\longrightarrow \Gamma(X,\mathscr{F}\otimes\mathscr{F})\longrightarrow\Gamma(X,\mathscr{F})
\end{equation*}
with $\Gamma(X,\mathscr{F})=R\Gamma(X,\QQ)$. Finally we take cohomology of these cochain complexes, using the Theorem of Kunneth, and we get a multiplication on $H(X,\QQ)$. This one is furthermore commutative. Here we recognize that the Godement resolution functor, the global section functor over $X$, and the cohomology functor of cochain complexes, are lax monoidal between symmetric monoidal categories.

We could also have started from any sheaf or complex of sheaves $\mathscr{F}$ with a structure of associative algebra $\mathscr{F}\otimes \mathscr{F} \rightarrow \mathscr{F}$ and then get such a multiplication on $H(X,\mathscr{F})$ defined via the cup-product map
\begin{equation*}
R\Gamma(X,\mathscr{F}) \otimes R\Gamma(X,\mathscr{F}) \longrightarrow R\Gamma(X,\mathscr{F}\otimes \mathscr{F})
\end{equation*}
composed with the multiplication on $\mathscr{F}$. This gives again a structure of associative algebra.

However, even if the structure of algebra on $\mathscr{F}$ were commutative (and even for $\mathscr{F}=\underline{\QQ}_X$) then $H(X,\mathscr{F})$ is a commutative algebra (in the DG sense) but not $R\Gamma(X,\mathscr{F})$; and if $\mathscr{F}$ is a sheaf of Lie algebras then $H(X,\mathscr{F})$ is a DG Lie algebra but not $R\Gamma(X,\mathscr{F})$.  One reason why the cup-product is not commutative at the level of cochains is that the Godement resolution functor is lax monoidal between symmetric monoidal categories but is \emph{not} compatible with the symmetric structures, whereas the other functors (global sections, cohomology) are.

\begin{definition}
Let $T$ be a lax monoidal functor as in Definition~\ref{definition:lax-monoidal} between the \emph{symmetric} monoidal categories  $(\mathcal{C},\otimes, 1_{\mathcal{C}})$ and $(\mathcal{D}, \otimes, 1_{\mathcal{D}})$. Recall that $s_{x,y}$ is the symmmetry isomorphism $x\otimes y\simeq y\otimes x$. Then $T$ is said to be \emph{symmetric} if for any two objects $x,y$ of $\mathcal{C}$ the following diagram is commutative in $\mathcal{D}$:
\begin{equation*}
\xymatrix{
T(x)\otimes T(y) \ar[r] \ar[d]_{s_{T(x),T(y)}} & T(x \otimes y) \ar[d]^{T\circ s_{x,y}} \\
T(y)\otimes T(x) \ar[r]  & T(y \otimes x) }
\end{equation*}
\end{definition}

Our interest into this notion lies in the fact that it preserves commutative algebras, as well as Lie algebras, by the argument that we used several times.

\begin{proposition}
\label{proposition:lax-symmetric-preserves-commutative}
If $T:(\mathcal{C},\otimes, 1_{\mathcal{C}})\rightarrow (\mathcal{D}, \otimes, 1_{\mathcal{D}})$ is a lax monoidal symmetric functor and $x$ is a commutative algebra in $\mathcal{C}$, with multiplication $\mu$ and unity $e:1_{\mathcal{C}}\rightarrow x$, then $T(x)$ is a commutative algebra in $\mathcal{D}$ with multiplication the composition
\begin{equation*}
T(x)\otimes T(x)\longrightarrow T(x\otimes x) \stackrel{T(\mu)}{\longrightarrow} T(x)
\end{equation*}
and unity the composition $1_\mathcal{D}\rightarrow T(1_\mathcal{C}) \stackrel{T(e)}{\longrightarrow} T(x)$.

Similarly, if $L$ is a Lie algebra in $\mathcal{C}$ with bracket $B$ then $T(L)$ is a Lie algebra in $\mathcal{D}$ with bracket the composition
\begin{equation*}
T(L)\otimes T(L) \longrightarrow T(L\otimes L) \stackrel{T(B)}{\longrightarrow} T(L) .
\end{equation*}
\end{proposition}

So, if instead of the Godement resolution we can use another resolution functor, inspired from rational homotopy theory, that is directly lax monoidal symmetric then we get commutative algebras or Lie algebras directly at the level of cochain complexes. This is what the next functor is going to do, introduced by Navarro Aznar for the use in Hodge theory.

\subsection{Thom-Whitney resolutions}
In this section $\kk$ needs to be a field of characteristic zero. The reference is \cite[\S~1--6]{Navarro}.

For any bounded-below complex of sheaves of $\kk$-vector spaces $\mathscr{F}$ on $X$ is constructed the \emph{Thom-Whitney resolution} $\mathscr{F}\rightarrow \TW^* \mathscr{F}$ which is a lax monoidal symmetric functor and is quasi-isomorphic to the usual Godement resolution. Using this, for a continuous map $f:X \rightarrow Y$ between topological spaces is constructed the \emph{Thom-Whitney direct image} $R_\TW f_{*}\mathscr{F}$ which is lax monoidal symmetric and quasi-isomorphic to the usual $Rf_*\mathscr{F}$. For $f$ the constant map to a point, $R_\TW f_* \mathscr{F}$ is also denoted by $R_\TW \Gamma(X, \mathscr{F})$ and is called the \emph{Thom-Whitney global sections}. In particular (Proposition~\ref{proposition:lax-symmetric-preserves-commutative}) if $\mathscr{F}$ is a sheaf of commutative DG algebras over $\kk$ then so is $R_\TW f_{*}\mathscr{F}$; and $R_\TW\Gamma(X, \underline{\QQ}_X)$ looks like a commutative DG algebra of rational polynomial differentials forms over~$X$ that appears in rational homotopy theory and does not involve any more cup-product-like formulas. So Thom-Whitney functors are a functorial version of these that apply to resolutions of all sheaves.

In practice, we just need to recall all the properties that we will use.

\begin{proposition}
The Thom-Whitney functors have the following properties:
\begin{enumerate}
\item $R_\TW f_*\mathscr{F}$ is functorial in $\mathscr{F}$.
\item $R_\TW f_*$ is naturally homotopy-equivalent to the classical $Rf_*$ defined using the Godement resolution.
\item $R_\TW f_*$ is a lax monoidal symmetric functor between the categories of complexes of sheaves over $X$ and $Y$.
\item $\mathscr{F}\rightarrow \TW^* \mathscr{F}$ is a soft resolution of $\mathscr{F}$.
\item $R_\TW f_*$ transforms quasi-isomorphisms of complexes of sheaves into quasi-isomorphisms.
\item The Thom-Whitney resolution also provides filtered resolutions. If $\mathscr{F}$ is equipped with a filtration, then $R_\TW f_*\mathscr{F}$ gets an induced filtration, and filtered quasi-isomorphism of complexes of sheaves are sent to filtered quasi-isomorphisms.
\item If $g:Y\rightarrow Z$ is a second continuous map of topological spaces the the functors $R_\TW (g\circ f)_*$ and
$R_\TW g_* \circ R_\TW f_*$ are naturally quasi-isomorphic. Similarly, for filtered complexes of sheaves, the functors
are filtered quasi-isomorphic.
\end{enumerate}
\end{proposition}

\subsection{Application to Hodge theory}
\label{subsection:application-Hodge}
From the point of view of Hodge theory, this has the following consequence. Let $X\subset \overline{X}$ be an open subset of a compact Kähler manifold or of a proper algebraic variety whose complement is a normal crossing divisor $D$. Let $j:X\hookrightarrow \overline{X}$. Navarro Aznar (\cite[\S~7--8]{Navarro}) shows that on $R_\TW j_* \underline{\QQ}_X$ there is a canonical filtration $\tau$ such that
\begin{equation*}
\bigl((R_\TW j_* \underline{\QQ}_X, \tau),\ (\Omega^\bullet_{\overline{X}}(\log D), W, F) \bigr)
\end{equation*}
is a mixed Hodge diagram of sheaves whose components are commutative DG algebras.
From the axioms of Thom-Whitney resolutions that we wrote above, it follows that applying $R_\TW\Gamma(\overline{X},-)$ to the above diagrams (i.e.\ to each component and to the filtered quasi-isomorphisms between them) gives a mixed Hodge diagram of commutative DG algebras, whose cohomology computes the cohomology of $X$ and puts a MHS on it. Furthermore such a diagram is functorial for morphisms, up to natural quasi-isomorphisms.

Let us sum up and relate to our introduction and \S~\ref{subsection:results}. We have an (admissible) VMHS $\mathbb{V}$ over $X$ for which we know that there exists a mixed Hodge complex $\operatorname{MHC}(X,\overline{X},\mathbb{V})$ that computes $H(X,\mathbb{V})$ and puts a MHS on it. We want to construct it using the Thom-Whitney functors \emph{in such a way that the functor
\begin{equation}
\mathbb{V} \longmapsto \operatorname{MHC}(X,\overline{X},\mathbb{V})
\end{equation}
is lax monoidal symmetric}, from the category of VMHS to the category of MHC. Then, if we apply this to a VMHS with a Lie bracket such as the adjoint local system $\ad_{\mathbb{V}}$ then the above MHC is automatically a mixed Hodge diagram of Lie algebras. And actually it is enough to construct a MHC of sheaves $\operatorname{\mathcal{MHC}}(X,\overline{X},\mathbb{V})$ over $\overline{X}$ such that $\mathbb{V}\mapsto \operatorname{\mathcal{MHC}}(X,\overline{X},\mathbb{V})$ is a lax monoidal symmetric functor (from VMHS to MHC of sheaves) since then we will simply put
\begin{equation}
\operatorname{MHC}(X,\overline{X},\mathbb{V}) := R_\TW\Gamma(\overline{X},\operatorname{\mathcal{MHC}}(X,\overline{X},\mathbb{V}))
\end{equation}
and this will be a MHC that is again (as a composition) lax monoidal symmetric as a functor of $\mathbb{V}$. For brevity we will simply say that we construct MHC's that are lax monoidal symmetric in $\mathbb{V}$.

\section{Non-degenerating VHS}
\label{section:non-degenerate}

In order to explain our ideas, we start by applying them to the case of a VHS $\mathbb{V}$ over a compact Kähler manifold $\overline{X}$. The VHS is defined over a field $\kk\subset \RR$, has some weight $w\in \ZZ$ and it is assumed to be polarizable.

\subsection{Compact base}
\label{section:VMHS-compact}
We first work over the compact base $X=\overline{X}$ itself and re-write the construction of our previous work. There is a pure Hodge structure of weight $w+i$ on $H^i(X, \mathbb{V})$, for all $i$, constructed by Deligne-Zucker. Define a filtration $F^\bullet$ on $\Omega^\bullet_X(\mathbb{V}):=\Omega^\bullet_X\otimes \mathbb{V}$ by
\begin{equation}
F^p(\Omega^\bullet_X(\mathbb{V})) = \sum_{r+e=p} F^r\Omega_X\otimes \mathcal{F}^e
\end{equation}
where $\mathcal{F}^\bullet$ is the filtration of $\mathbb{V}\otimes\mathcal{O}_X$ by holomorphic sub-vector bundles.
In other words the classes of Hodge type $(p,q)$ in $H^i(X, \mathbb{V})$ will come from differential forms of type $(r,s)$
with values in $\mathcal{V}^{e,f}$ (the $\mathcal{C}^\infty$ vector bundle of Hodge type $(e,f)$) for $(r+e,s+f)=(p,q)$.
The statement is then:
\begin{theorem}[{\cite[Thm.~2.9, Cor.~2.11]{Zucker}}]
If $X$ is compact Kähler and $\mathbb{V}$ is a polarizable VHS of weight $w$ then
\begin{equation*}
\bigl(\mathbb{V},\  (\Omega^\bullet_X(\mathbb{V}), F) \bigr)
\end{equation*}
is a pure Hodge complex of sheaves of weight $w$, which is lax monoidal symmetric in~$\mathbb{V}$.
\end{theorem}

\begin{proof}
The cited theory of Deligne-Zucker already shows that this is a mixed Hodge complex. For two polarizable VHS $\mathbb{V}_1,\mathbb{V}_2$ there is a diagram
\begin{equation*}
\xymatrix{
\Omega^\bullet_X(\mathbb{V}_1) \otimes \Omega^\bullet_X(\mathbb{V}_2) \ar[r] & \Omega^\bullet_X(
\mathbb{V}_1 \otimes_{\kk} \mathbb{V}_2) \\
(\mathbb{V}_1 \otimes_{\kk} \CC) \otimes (\mathbb{V}_2 \otimes_{\kk} \CC) \ar[r] \ar[u]^{\approx} &
(\mathbb{V}_1 \otimes \mathbb{V}_2) \otimes_{\kk} \CC \ar[u]_{\approx} }
\end{equation*}
where the two vertical morphisms are quasi-isomorphisms (this is simply the holomorphic Poincaré lemma extended to local systems) and the top morphism is induced by the product in $\Omega^\bullet_X$. This diagram is compatible with the symmetry exchanging $\mathbb{V}_1$ and $\mathbb{V}_2$. Whence the lax monoidal symmetric condition.
\end{proof}

As explained in \S~\ref{subsection:application-Hodge} this the is MHC of sheaves, to which we apply $R_\TW\Gamma(X,-)$ to get a MHC that is lax monoidal symmetric in $\mathbb{V}$.

Of course this is just re-writing \cite[\S~10]{Lefevre2} using only $R_\TW\Gamma$ instead of, over $\CC$, the more classical resolution by $\mathcal{C}^\infty$ differential forms with values in $\mathbb{V}$.

Also, this is easily extended to graded-polarizable VMHS. If $\mathbb{W}_\bullet$ is the filtration of such a VMHS $\mathbb{V}$ by sub-local systems such that the quotients are polarizable VHS, then the filtration $W_\bullet$ on $H^i(X,\mathbb{V})$ is such that graded quotients are the Deligne-Zucker Hodge structures $H^i(X, \Gr^W_k(\mathbb{V}))$ of weight $k+i$. Hence the weight filtration on the MHC of sheaves is defined by
\begin{equation}
W_k(\Omega^\bullet_X(\mathbb{V})) := \Omega^\bullet_X(\mathbb{W}_k)
\end{equation}
and then:

\begin{corollary}
If $\mathbb{V}$ is a graded-polarizable VMHS then
\begin{equation*}
\bigl((\mathbb{V},W),\  (\Omega^\bullet_X(\mathbb{V}), W, F) \bigr)
\end{equation*}
is a MHC of sheaves, lax monoidal symmetric in $\mathbb{V}$.
\end{corollary}

\subsection{VHS extendable to a compactification}
\label{section:VHS-extendable}

Now we turn to the situation where the VHS $\mathbb{V}$ is defined over $\overline{X}$ but we study it over the open variety $X$, which embeds as $j:X\hookrightarrow \overline{X}$ with complement a normal crossing divisor $D$. After a theorem of Griffiths, this is the same as requiring that $\mathbb{V}$ is a VHS over $X$ with trivial monodromy around $D$: then the local system underlying $\mathbb{V}$ extends across $D$ and Griffiths' theorem \cite[Theorem~9.5]{GriffithsIII} shows that the holomorphic bundles $\mathcal{F}^\bullet$ also extend so as to define a VHS on $\overline{X}$. In this case there is a MHS on each $H^i(X, \mathbb{V})$ with lowest weight-piece the previous Deligne-Zucker Hodge structure $H^i(\overline{X}, \mathbb{V})$ of weight $w+i$.

Over $X$ we can introduce the logarithmic complex $\Omega^\bullet_{\overline{X}}(\log D)$ which carries two filtrations $W,F$ as usual and we can form the logarithmic complex with local coefficients $\Omega^\bullet_{\overline{X}}(\log D,\mathbb{V}):=\Omega^\bullet_{\overline{X}}(\log D)\otimes\mathbb{V}$. It again carries two filtrations $F,W$ and since $\mathbb{V}$ is pure $W$ is simply the shift by $w$ of the weight filtration of $\Omega^\bullet_{\overline{X}}(\log D)$. This will be the complex part of the MHC of sheaves.

It is easy to see that we again have a diagram
\begin{equation*}
\xymatrix{
\Omega^\bullet_{\overline{X}}(\log D,\mathbb{V}_1) \otimes \Omega^\bullet_{\overline{X}}(\log D,\mathbb{V}_2) \ar[r] & \Omega^\bullet_{\overline{X}}(\log D, \mathbb{V}_1 \otimes_{\kk} \mathbb{V}_2) \\
(\mathbb{V}_1 \otimes\CC) \otimes (\mathbb{V}_2 \otimes\CC) \ar[r] \ar[u]^{\approx} &
(\mathbb{V}_1 \otimes_{\kk} \mathbb{V}_2) \otimes \CC \ar[u]_{\approx} }
\end{equation*}
where the vertical maps are quasi-isomorphisms and the top horizontal one uses the product of differential forms.

For the component over $\kk$ we take $R_\TW j_* \mathbb{V}$.
The filtration is obtained from the canonical filtration $\tau$. Recall that for any complex $(K^\bullet,d)$ the canonical filtration $\tau$ is defined by $\tau_k K^n=0$ if $k<n$, $\tau_k K^n=\Ker(d)$ if $k=n$ and $\tau_k K^n=K^n$ for $k>n$. Any quasi-isomorphism is then automatically a filtered quasi-isomorphism with respect to $\tau$. On $R_\TW j_* \mathbb{V}$ we take for $W$ the shifted filtration $\tau[w]$, with $\tau[w]_k:=\tau_{w+k}$.

\begin{proposition}[{Compare \cite[Prop.~8.4]{Navarro}}]
\label{proposition:canonical-chain-quasi-isomorphisms}
There is a canonical chain of filtered quasi-isomorphisms
\begin{equation*}
(R_\TW j_* \mathbb{V}, W) \otimes \CC \stackrel{\approx}{\longleftrightarrow} (\Omega^\bullet_{\overline{X}}(\log D,\mathbb{V}), W) .
\end{equation*}
\end{proposition}

\begin{proof}
First, by the holomorphic Poincaré lemma there is a quasi-isomorphism of sheaves over $X$
\begin{equation*}
\mathbb{V}\otimes\CC \stackrel{\approx}{\longrightarrow} \Omega^\bullet_X(\mathbb{V})
\end{equation*}
that we compose with the Thom-Whitney resolution
\begin{equation*}
\Omega^\bullet_X(\mathbb{V}) \stackrel{\approx}{\longrightarrow} \TW^* \Omega^\bullet_X(\mathbb{V})
\end{equation*}
to get the chain
\begin{equation}
R_\TW j_*\mathbb{V}\otimes\CC 
 \stackrel{\approx}{\longrightarrow} R_\TW j_*\Omega^\bullet_X(\mathbb{V}) \stackrel{\approx}{\longrightarrow}
R_\TW j_*(\TW^* \Omega^\bullet_X(\mathbb{V})).
\end{equation}
These are automatically filtered quasi-isomorphisms for $\tau$ as well as for $\tau[w]$.

On the other hand there is a canonical morphism over $\overline{X}$
\begin{equation*}
\Omega^\bullet_{\overline{X}}(\log D,\mathbb{V}) \longrightarrow j_*\Omega^\bullet_X(\mathbb{V})
\end{equation*}
which, composed with $j_*(\Omega^\bullet_X(\mathbb{V}))\rightarrow j_*(\TW^* \Omega^\bullet_X(\mathbb{V}))$, induces a quasi-isomorphism
\begin{equation*}
\Omega^\bullet_{\overline{X}}(\log D,\mathbb{V}) \stackrel{\approx}{\longrightarrow} j_*(\TW^*\Omega^\bullet_X(\mathbb{V})) .
\end{equation*}
Indeed, $\Omega^\bullet_X(\mathbb{V})\rightarrow\TW^*\Omega^\bullet_X(\mathbb{V})$ is an acyclic resolution so both terms compute
$Rj_* \mathbb{V}$.
Then we use the quasi-isomorphism
\begin{equation*}
j_*(\TW^*\Omega^\bullet_X(\mathbb{V})) \stackrel{\approx}{\longrightarrow} R_\TW j_*(\TW^*\Omega^\bullet_X(\mathbb{V})) .
\end{equation*}
To sum up this gives
\begin{equation}
\Omega^\bullet_{\overline{X}}(\log D,\mathbb{V})
\stackrel{\approx}{\longrightarrow} j_*(\TW^*\Omega^\bullet_X(\mathbb{V}))
 \stackrel{\approx}{\longrightarrow} R_\TW j_*(\TW^*\Omega^\bullet_X(\mathbb{V}))
\end{equation}
where again these quasi-isomorphisms are filtered quasi-isomorphisms for $\tau$ and so also for~$\tau[w]$.

Finally the classical filtered quasi-isomorphism
\begin{equation*}
(\Omega^\bullet_{\overline{X}}(\log D), \tau) \stackrel{\approx}{\longrightarrow}
(\Omega^\bullet_{\overline{X}}(\log D), W)
\end{equation*}
also induces with local coefficients
\begin{equation*}
(\Omega^\bullet_{\overline{X}}(\log D), \tau)\otimes\mathbb{V} \stackrel{\approx}{\longrightarrow}
(\Omega^\bullet_{\overline{X}}(\log D), W)\otimes\mathbb{V} .
\end{equation*}
By our choice of $W=\tau[w]$ over $\kk$ this gives
\begin{equation}
(\Omega^\bullet_{\overline{X}}(\log D,\mathbb{V}), \tau[w]) \stackrel{\approx}{\longrightarrow}
(\Omega^\bullet_{\overline{X}}(\log D, \mathbb{V}), W) .
\qedhere
\end{equation}
\end{proof}

Then it's just about applying the classical theory of residues of forms with logarithmic poles, taking into account local coefficients in a VHS. We decompose the divisor $D$ (assuming here \emph{simple} normal crossings) into its components $D=\bigcup_I D_i$ ($I$ a finite set), for each set $J\subset I$ we write $D_J:=\cap_{i\in J} D_i$ and we let $\widetilde{D}^k := \bigcup_{|J|=k} D_J$ with its canonical inclusion $i_k$ into $\overline{X}$. For $k=0$ this is $\overline{X}$. The classical residue is a morphism
\begin{equation}
\Res:W_k\Omega_{\overline{X}}^\bullet(\log D) \longrightarrow (i_k)_* \Omega^{\bullet-k}_{\widetilde{D}^k}
\end{equation}
which induces a quasi-isomorphism
\begin{equation}
\Gr_k^W \Omega_{\overline{X}}^\bullet(\log D) \stackrel{\approx}{\longrightarrow} (i_k)_* \Omega^{\bullet-k}_{\widetilde{D}^k} .
\end{equation}
On the right-hand side the terms $\widetilde{D}^k$ are compact Kähler manifolds with pure Hodge structures on cohomology. Shifting their weights, this quasi-isomorphism puts a pure Hodge structure of weight $k+n$ on $H^n(\Gr_k^W \Omega_{\overline{X}}^\bullet(\log D))$ and allows to check all the axioms of mixed Hodge complex.

We easily conclude for our non-degenerating VHS by defining a residue morphism with local coefficients
\begin{equation}
\label{equation:residue-local-coefficients}
\Res\otimes\id_\mathbb{V}:W_{k+w}\Omega_{\overline{X}}^\bullet(\log D, \mathbb{V}) \longrightarrow (i_k)_* \Omega^{\bullet-k}_{\widetilde{D}^k}(\mathbb{V}) .
\end{equation}

\begin{theorem}
For a polarizable VHS $\mathbb{V}$ defined over $X$ and extending as VHS to the compactification $j:X\hookrightarrow \overline{X}$, the data
\begin{equation*}
\bigl( (R_\TW j_*\mathbb{V}, W),\ (\Omega^\bullet_{\overline{X}}(\log D,\mathbb{V}), W, F) \bigr)
\end{equation*}
is a MHC of sheaves for $H^\bullet(X, \mathbb{V})$ that is lax monoidal symmetric in $\mathbb{V}$.
\end{theorem}

\begin{proof}
Via the map~\eqref{equation:residue-local-coefficients},
the terms $H^n(\Gr_{k+w}^W \Omega_{\overline{X}}^\bullet(\log D))$ carry pure Hodge structures of weight $n+k+w$, because on the right-hand side we get cohomology groups of compact Kähler manifolds with coefficients in the polarizable VHS $\mathbb{V}$ of weight~$w$. Thus this forms a MHC, and in Proposition~\ref{proposition:canonical-chain-quasi-isomorphisms} every term is lax monoidal symmetric in~$\mathbb{V}$.
\end{proof}

In the notation of \S~\ref{subsection:application-Hodge} the above is $\operatorname{\mathcal{MHC}}(X, \overline{X}, \mathbb{V})$, and then we will apply $R_\TW\Gamma(\overline{X},-)$ to it.

We see from this construction that each $H^i(X,\mathbb{V})$ has weights between $i+w$ and $2i+w$, with $\Gr^W_{i+w}H^i(X,\mathbb{V})=H^i(\overline{X},\mathbb{V})$.

\section{Degenerating VMHS on curves}
\label{section:curves}
We now come to the study of degenerating VHS and even VMHS, which is more complicated. So we apply it only for curves in this section.

\subsection{Preliminaries on extensions of local systems}
In this section we work with any open manifold $X\subset\overline{X}$ compactified by a normal crossing divisor~$D$.

Let $\mathbb{V}$ be a local system over $X$ of vector spaces over $\CC$. Deligne has shown (\cite[\S~II.5]{DeligneRegSing}, see also \cite[\S~5]{Hotta}) that its associated flat vector bundle $(\mathcal{V},\nabla)$ extends as a meromorphic bundle $\mathcal{M}$ to $\overline{X}$. For each half-open interval $I$ of length $1$, there is a unique vector bundle with connection $(\overline{\mathcal{V}}^I,\nabla)$ over $\overline{X}$, seen as a lattice inside $\mathcal{M}$, extending $(\mathcal{V},\nabla)$ such that $\nabla$ has logarithmic poles on $D$ and such that the residue of $\nabla$ along any (local) component $D_i$ of $D$
\begin{equation*}
\Res_{D_i}(\nabla)\in\Omega_{D_i}^1(\End(\overline{\mathcal{V}}))
\end{equation*}
has eigenvalues of real part contained in~$I$.

For two such flat bundles $(\mathcal{V}_1, \nabla_1)$, $(\mathcal{V}_2, \nabla_2)$ with associated meromorphic extensions $\mathcal{M}_1, \mathcal{M}_2$ the meromorphic extension of $\mathcal{V}_1 \otimes \mathcal{V}_2$, equipped with the connection
\begin{equation}
\nabla:={\nabla_1} \otimes {\id_2} + {\id_1} \otimes {\nabla_2},
\end{equation}
is exactly $\mathcal{M}_1 \otimes \mathcal{M}_2$. However the construction of $\overline{\mathcal{V}}^I$ is not compatible with this tensor product.

Hence we will make the assumption that our local systems have \emph{unipotent} monodromy around each component $D_i$. In this case we can take for $(\overline{\mathcal{V}}, \nabla)$ the unique extension for which the residues are nilpotent (i.e.\ have $0$ as only eigenvalue) and call it the \emph{canonical extension} of $(\mathcal{V},\nabla)$. Note that the residue of a tensor product $(\mathcal{V},\nabla)=(\mathcal{V}_1,\nabla_1)\otimes (\mathcal{V}_2,\nabla_2)$ is
\begin{equation}
\Res_{D_i}(\nabla)={\Res_{D_i}(\nabla_1)}\otimes{\id_2} + {\id_1} \otimes {\Res_{D_i}(\nabla_2)}
\end{equation}
which is again nilpotent if both residues are.

This will be a good tool to define our lax symmetric monoidal MHC:

\begin{proposition}
\label{theorem:chain-qis}
There are two chains of quasi-isomorphisms
\begin{equation}
\label{equation:chain-qis-1}
R_\TW j_*\mathbb{V}
\stackrel{\approx}{\longrightarrow} R_\TW j_*\Omega^\bullet_X(\mathbb{V}) \stackrel{\approx}{\longrightarrow}
R_\TW j_*(\TW^* \Omega^\bullet_X(\mathbb{V}))
\end{equation}
and
\begin{equation}
\label{equation:chain-qis-2}
\Omega^\bullet_{\overline{X}}(\log D,\overline{\mathcal{V}})
\stackrel{\approx}{\longrightarrow} j_*(\TW^*\Omega^\bullet_X(\mathbb{V}))
\stackrel{\approx}{\longrightarrow} R_\TW j_*(\TW^*\Omega^\bullet_X(\mathbb{V}))
\end{equation}
that are lax symmetric monoidal for local systems with unipotent monodromy around~$D$.
\end{proposition}

\begin{proof}
Same as Proposition~\ref{proposition:canonical-chain-quasi-isomorphisms}; the only difference is that it is now required to show that the map
\begin{equation*}
\Omega^\bullet_{\overline{X}}(\log D,\overline{\mathcal{V}}) \longrightarrow j_*(\TW^*\Omega^\bullet_X(\mathbb{V}))
\end{equation*}
is a quasi-isomorphism. Both terms compute $Rj_* \mathbb{V}$ (\cite[Theorem~5.2.24]{Hotta}).
\end{proof}

\subsection{Degenerations}
Now we fix a smooth proper curve $\overline{X}$ (automatically Kähler), with $D$ a finite set of points and $X=\overline{X}\setminus D$. It is known that VMHS of geometric origin automatically have quasi-unipotent monodromy around $D$, so the pull-back to a finite cover has unipotent monodromy. In thi section we will restrict ourselves to local unipotent monodromy.

In his study of degenerations of polarizable VHS, Schmid \cite{Schmid} has shown that the holomorphic bundles $\mathcal{F}^\bullet$ of a polarizable VHS $\mathbb{V}$ of weight $w$ on $X$ can be extended to $\overline{X}$ as holomorphic sub-vector bundles of $\overline{\mathcal{V}}$. Around $x\in D$ the residue operator is a nilpotent operator $N(x)$ of the fiber $\overline{\mathcal{V}}_x$ to which is associated a monodromy filtration:

\begin{definition}[{\cite[2.1]{SteenbrinkZucker}}]
The \emph{monodromy filtration} associated to a finite-dimensional vector space $V$ and a nilpotent endomorphism $N$, centered at $w\in\ZZ$, is the unique increasing filtration $M_\bullet$ such that for all $i\in\ZZ$:
\begin{enumerate}
\item $N(M_i)\subset M_{i-2}$,
\item $N^i$ induces an isomorphism $\Gr^{w+i}_M V \stackrel{\simeq}{\longrightarrow} \Gr^{w-i}_M V$.
\end{enumerate}
\end{definition}

Schmid then shows that the this monodromy filtration centered at $w$ together with $\overline{\mathcal{F}}^\bullet_x$ gives a MHS on~$\overline{\mathcal{V}}_x$.

In this situation Zucker \cite[\S~13]{Zucker} constructed, first a pure HS of weight $i+w$ on each $H^i(\overline{X}, j_* \mathbb{V})$, and then a MHS on $H^i(X, \mathbb{V})$
whose part of lowest weight is exactly $H^i(\overline{X}, j_*\mathbb{V})$. The Hodge filtration is defined on $\Omega^\bullet_{\overline{X}}(\log D, \overline{\mathcal{V}})$ as usual using the bundles $\overline{\mathcal{F}}^\bullet \subset \overline{\mathcal{V}}$. The description of the weight filtration uses the previous monodromy filtration; in particular the upper bound for the weights of $H^i(X, \mathbb{V})$ depends on the order of nilpotency of $N(x)$. The MHS constructed by Zucker comes indeed from a MHC and is defined over the field $\kk\subset\RR$ if $\mathbb{V}$ is.

Now when $\mathbb{V}$ is a graded-polarizable VMHS, then the local systems $\mathbb{W}_\bullet\subset\mathbb{V}$ also extends as vector bundles to $\overline{\mathcal{W}}_\bullet \subset\overline{\mathcal{V}}$ and the previous theory applies to the weight-graded local systems that are by definition polarizable VHS; however this is not enough. The problem of getting a similar good theory of degenerations is studied in detail in the paper of Steenbrink-Zucker and there is introduced the notion of \emph{admissible} VMHS (\cite[3.13]{SteenbrinkZucker}). For this are required graded-polarizability and local quasi-unipotent monodromy, that the filtration $\mathcal{F}^\bullet\subset\mathcal{V}$ also extend to $\overline{\mathcal{F}}^\bullet \subset\overline{\mathcal{V}}$, and that the nilpotent residue $N(x)$ around $x$ behaves well with respect to the weight filtration of each fiber $\overline{\mathcal{V}}_x$, i.e.\ there exists the relative monodromy filtration:

\begin{definition}[{\cite[2.5]{SteenbrinkZucker}}]
The \emph{relative monodromy filtration} (it is unique when it exists) associated to a finite-dimensional vector space $V$ with an increasing filtration $W_\bullet$ and a nilpotent endomorphism $N$ respecting $W_\bullet$ is a filtration $M_\bullet$ of $V$ such that:
\begin{enumerate}
\item For all $i\in\ZZ$, $N(M_i)\subset M_{i-2}$,
\item $M$ induces on each $\Gr^k_W V$ the monodromy filtration of the endomorphism induced by $N$, centered at $k$.
\end{enumerate}
\end{definition}

Polarizable VHS are admissible because of the results of Schmid; the relative monodromy filtration is simply the monodromy filtration of $N$. Also, VMHS of geometric origin are shown to be admissible in their paper (\S~5). For admissible VMHS, they are then able to construct a MHS on $H^i(X, \mathbb{V})$, again coming from a MHC and defined over $\kk\subset\RR$ if $\mathbb{V}$ is.

Since we work on tensor product, we will need to know:

\begin{proposition}[{\cite[Appendix]{SteenbrinkZucker}}]
For two finite-dimensional vector spaces $V_1,V_2$ with nilpotent endomorphisms $N_1,N_2$ and filtrations $W^1_\bullet, W^2_\bullet$, having relative monodromy filtrations $M^1_\bullet, M^2_\bullet$, then the monodromy filtration of $V_1 \otimes V_2$ with respect to the endomorphism $N={N_1}\otimes{\id_{V_2}} + {\id_{V_1}}\otimes{N_2}$ and the filtration $W^1 \otimes W^2$ exists and is $M^1 \otimes M^2$. Consequently, the category of admissible VMHS on a curve is closed under tensor product.
\end{proposition}

\subsection{Construction for admissible VMHS}

Under the previous hypothesis for $X$ and for an admissible VMHS, we follow the construction of Steenbrink-Zucker but we make explicit the lax symmetric monoidal dependence on $\mathbb{V}$.

\begin{theorem}[{\cite[\S~4]{SteenbrinkZucker}}]
\label{theorem:SteenbrinkZucker-multiplicative}
For an admissible VMHS $\mathbb{V}$ over a curve $j:X\hookrightarrow \overline{X}$ with unipotent monodromy around the singularities, defined over a field $\kk\subset\RR$, there exists a filtration $\mathfrak{W}_\bullet$ on all terms of Theorem~\ref{theorem:chain-qis} such that the data
\begin{equation*}
\bigl( (R_\TW j_*\mathbb{V}, \mathfrak{W}),\ (\Omega^\bullet_{\overline{X}}(\log D,\overline{\mathcal{V}}), \mathfrak{W}, F) \bigr)
\end{equation*}
is a MHC of sheaves, lax monoidal symmetric for those VMHS with unipotent monodromy.
\end{theorem}

Recall the classical construction.
At each point $x$ of $D$, on the fiber $\overline{\mathcal{V}}_x$, is defined a filtration $\overline{\mathcal{W}}_\bullet(x)$ (coming from the canonical extensions $\overline{\mathcal{W}}_\bullet$ of the local systems $\mathbb{W}_\bullet$ with associated flat bundles $(\mathcal{W}_\bullet, \nabla)$), a nilpotent endomorphism $N(x)$ (residue of $\nabla$), and the relative monodromy filtration $M_\bullet(x)$ whose existence is part of the admissibility conditions. Then is defined a filtration
\begin{equation}
Z_k(x):=N(x)(\overline{\mathcal{W}}_k(x))+M_{k-1}(x)\cap\overline{\mathcal{W}}_{k-1}(x) .
\end{equation}
This defines uniquely a filtration $\mathfrak{W}_\bullet$ of $\Omega^\bullet_{\overline{X}}(\log D,\overline{\mathcal{V}})$, where $\mathfrak{W}_k \Omega^\bullet_{\overline{X}}(\log D,\overline{\mathcal{V}})$ is the sub-complex of $\Omega^\bullet_{\overline{X}}(\log D,\overline{\mathcal{W}}_k)$ formed by $\overline{\mathcal{W}}_k$ in degree $0$ and whose part in degree $1$ at $x$ lies between $\Im(\nabla)$ and $\Omega_{\overline{X}}^1(\log D,\overline{\mathcal{W}}_k)$ and is sent to $Z_k(x)$ via the residue
\begin{equation*}
\Res_x:\Omega^1_{\overline{X}}(\log D,\overline{\mathcal{W}}_k) \longrightarrow \overline{\mathcal{W}}_k(x).
\end{equation*}
For a local coordinate $t$ at $x$, with $Z_k(x)$ extending to a local system $\mathcal{Z}_k$ and a sub-bundle $\overline{\mathcal{Z}}_k$ of $\overline{\mathcal{W}}_k$, the complex $\mathfrak{W}_k \Omega^\bullet_{\overline{X}}(\log D,\overline{\mathcal{V}})$ is
\begin{equation}
\label{equation:local-coordinate-W-SteenbrinkZucker}
\overline{\mathcal{W}}_k \stackrel{d}{\longrightarrow}
\frac{dt}{t}\otimes (\overline{\mathcal{Z}}_k + t\overline{\mathcal{W}}_k) .
\end{equation}

On $Rj_* \mathbb{V}$ over $\kk$, for any functorial resolution $Rj_*$, this corresponds to the following construction: $\mathfrak{W}_k Rj_* \mathbb{V}$ is the sub-complex of $Rj_*\mathbb{V}$ formed by
\begin{enumerate}
\item First consider $Rj_*(\mathbb{W}_k) \subset Rj_* \mathbb{V}$.
\item Then truncate: $(\tau_{\leq 1}Rj_*(\mathbb{W}_k))^0=Rj_*(\mathbb{W}_k)^0$,
\begin{equation*}
(\tau_{\leq 1}Rj_*(\mathbb{W}_k))^1=\Ker\bigl(d:Rj_*(\mathbb{W}_k)^1\longrightarrow Rj_*(\mathbb{W}_k)^2\bigr)
\end{equation*}
and $(\tau_{\leq 1}Rj_*(\mathbb{W}_k))^i=0$ for $i>1$. It has the same cohomology as $Rj_*(\mathbb{W}_k)$ since on a curve $R^i j_*(\mathbb{W}_k)=0$ for $i>1$.
\item Finally $(\mathfrak{W}_k Rj_* \mathbb{V})^0:=(Rj_*(\mathbb{W}_k))^0$ and
\begin{equation*}
(\mathfrak{W}_k Rj_* \mathbb{V})^1 := d(Rj_*(\mathbb{W}_k)^0) + (\tau_{\leq 1}Rj_*(\mathbb{W}_k))^1 \cap Rj_*(\mathcal{Z}_k)^1 .
\end{equation*}
\end{enumerate}
The computations of $\Gr^{\mathfrak{W}}_\bullet$ done in their article indeed shows that this gives a MHC of sheaves.

\begin{proof}[{Proof of the lax symmetric monoidal condition in Theorem~\ref{theorem:SteenbrinkZucker-multiplicative}}]
In our situation, we apply this with $Rj_*$ replaced by any of the functors appearing in Proposition~\eqref{theorem:chain-qis}, that is, with $R_\TW$, $R_\TW j_* \Omega_X^\bullet$, $R_\TW j_*(\TW^* \Omega_X^\bullet)$, $j_*(\TW^* \Omega_X^\bullet)$. These are defined over $\kk$ when $\mathbb{V}$ is.

Given two admissible VMHS $\mathbb{V}_1$, $\mathbb{V}_2$ and a point $x\in D$, with associated residues $N_1(x), N_2(x)$, monodromy filtrations $M^1_\bullet(x),M^2_\bullet(x)$, filtrations $Z^1_\bullet(x),Z^2_\bullet(x)$ etc, we have natural symmetric maps
\begin{equation*}
R_\TW j_* \mathbb{V}_1 \otimes R_\TW j_* \mathbb{V}_2 \longrightarrow R_\TW j_* (\mathbb{V}_1 \otimes \mathbb{V}_2)
\end{equation*}
and
\begin{equation*}
\Omega^\bullet_{\overline{X}}(\log D,\overline{\mathcal{V}}_1) \otimes \Omega^\bullet_{\overline{X}}(\log D,\overline{\mathcal{V}}_2) \longrightarrow \Omega^\bullet_{\overline{X}}(\log D,\overline{\mathcal{V}_1 \otimes \mathcal{V}_2})
\end{equation*}
(and similarly for all the other functors $Rj_*$) and this last one is compatible with the filtrations $F^\bullet$. What remains to check is the compatibility with this weight filtration~$\mathfrak{W}$.

Since our complexes are concentrated in degrees $0$ and $1$ the only non-trivial multiplication is the one between degrees $0$ and $1$. If we denote by $M_\bullet(x)$ the relative monodromy filtration of $\mathbb{V}_1 \otimes \mathbb{V}_2$ at $x$, for the residue $N(x)$ with its associated $Z_\bullet$(x), all we have to do (explicit for example in equation~\eqref{equation:local-coordinate-W-SteenbrinkZucker}) is to check that there is a natural inclusion
\begin{equation}
\overline{\mathcal{W}}^1_k(x) \otimes Z^2_\ell(x) \subset Z_{k+\ell}(x)
\end{equation}
for all $k,\ell$.

This is linear algebra; to compute this let us drop the letter $x$ and write $W_\bullet$ for $\overline{\mathcal{W}}_\bullet$. Recall that $N={N_1}\otimes{\id_2}+{\id_1}\otimes{N_2}$, $M=M^1 \otimes M^2$, $W=W^1 \otimes W^2$ and $Z_\ell=N(W_\ell)+M_{\ell-1}\cap W_{\ell-1}$. But it will be more practical to have another expression for $Z_\bullet$ as proved in \cite[\S~3.4]{Kashiwara}:
\begin{equation}
Z_\ell = N(W_\ell) + \bigcap_{j\geq 0} (N^j)^{-1}(W_{\ell-j}) .
\end{equation}
So, to compute $W_k^1 \otimes Z^2_\ell$, for one term of the sum
\begin{equation*}
W^1_k \otimes N_2(W^2_\ell) \subset ({N_1}\otimes{\id_2}+{\id_1}\otimes{N_2)({W^1_k} \otimes {W^2_\ell}})  \subset N(W_{k+\ell}) \subset Z_{k+\ell}
\end{equation*}
and for the other term
\begin{multline*}
W^1_k \otimes \bigcap_{j\geq 0} (N_2^j)^{-1}(W^2_{\ell-j}) \\
\subset \bigcap_{j\geq 0} \bigl(({N_1}\otimes{\id_2}+{\id_1}\otimes{N_2})^j\bigr)^{-1}(W^1_k \otimes W^2_{\ell-j}) \\
\subset \bigcap_{j\geq 0} (N^j)^{-1}(W_{k+\ell-j}) \subset Z_{k+\ell} .
\end{multline*}
This concludes.
\end{proof}

In this construction, if $\mathbb{V}$ has weights greater than $w$ then $H^i(X,\mathbb{V})$ has weights greater than $w+i$ for all~$i$.

\begin{remark}
If $\mathbb{V}$ is a pure VHS of weight $w$, then in the above construction the filtration $W_\bullet$ is concentrated in weight $w$ and $M_\bullet$ is simply the monodromy filtration of $N$, so the formula for $Z$ is
\begin{equation*}
Z_k(x)=N(x)(\overline{\mathcal{V}}_x)+M_{k-1}(x), \quad w\leq k \leq w+q+1
\end{equation*}
(where $q$ is the order of nilpotency of $N(x)$), with $Z_k(x)=0$ or $\overline{\mathcal{V}}_x$ outside of this range.
This is the original construction of \cite[\S~13]{Zucker}, here with $R_\TW j_*$.
\end{remark}

\begin{remark}
If furthermore there is a point $x$ of $D$ over which $\mathbb{V}$ is non-degenerate (i.e.\ extends as VHS, equivalently has trivial monodromy) then the following happens: the monodromy is trivial, the canonical extension of $\mathbb{V}$ is $\mathbb{V}$ itself, the operator $N(x)$ is zero, the monodromy filtration is trivial and the weight filtration is concentrated in $w$. One can then check line by line that everything reduces to the case of our section~\ref{section:VHS-extendable}.
\end{remark}

\section{Mixed Hodge modules}
\label{section:MHM}

To deal with the most general case, we \emph{need} to use the theory of mixed Hodge modules of Saito. This actually constructs a whole category $\MHM(X)$ for algebraic varieties $X$ with a six-functor formalism and vanishing and nearby cycles, enhancing the category of perverse sheaves with mixed Hodge structures, and such that a VMHS on an open dense subset of $X$ defines an object of $\MHM(X)$. The MHS on cohomology then comes from the push-forward to the point; tracking down the construction gives indeed a mixed Hodge complex.

This is contained in two articles \cite{SaitoMH}, \cite{SaitoMHM}. See also the survey of Schnell \cite{Schnell}. Since we focus only on the construction of MHS on cohomology with local coefficients, the survey of Arapura \cite{Arapura} as well as the short note \cite{SaitoAdmissible} are very useful.

For this to work there is again a notion of admissibility for a VMHS on a complex manifold $X$ relatively to a compactification $\overline{X}$, studied by Kashiwara in \cite{Kashiwara}. It is given by testing the admissibility of Steenbrink-Zucker on pull-backs to curves; it includes in particular the quasi-unipotency of the monodromy around the boundary. Note that it depends on the compactification but still it is an invariant of codimension~$2$, hence for smooth algebraic varieties it is independent because all smooth compactifications are birationally equivalent.

\subsection{Brief reminder}
For any complex manifold $X$ of some dimension $N$ is defined a category $\MF(X)$ of \emph{filtered $\mathcal{D}_X$-modules with $\kk$-structure} ($\kk \subset \RR$) whose objects are triples $(P,\mathcal{M},\alpha)$ where $P$ is a perverse sheaf over $\kk$ on $X$,
$\mathcal{M}$ is a regular holonomic $\mathcal{D}_X$-module with a good filtration $\mathcal{F}^\bullet$, and $\alpha$ is a quasi-isomorphism $\DR(\mathcal{M})\simeq P\otimes\CC$ where
\begin{equation}
\DR(\mathcal{M}) := \bigl\{ \mathcal{M} \longrightarrow \Omega^1_X \otimes \mathcal{M} \longrightarrow \cdots \longrightarrow \Omega^N_X \otimes \mathcal{M} \bigr\}[N] .
\end{equation}
(the leftmost term is placed in degree~$-N$).
Any VHS $(\mathbb{V}, \mathcal{F}^\bullet, \nabla)$ over $X$ defines an object of $\MF(X)$ where $P$ is $\mathbb{V}[N]$ (shifting the local system over $\kk$ by the dimension of $X$ to make it a perverse sheaf on $X$) and $\mathcal{M}$ is the $\mathcal{D}_X$-module coming from the flat connection $(\mathbb{V}\otimes\mathcal{O}_X,\nabla)$. Griffiths' transversality is then precisely the statement that $\mathcal{F}^\bullet$ forms a filtration of $\mathcal{M}$ as $\mathcal{D}_X$-module.

For each $w\in\ZZ$, the category of \emph{pure polarizable Hodge modules of weight $w$} is a full sub-category $\MH(X,w)\subset \MF(X)$. The axioms are quite evolved; they imply that each element of $\MH(X,w)$ has a \emph{decomposition by strict support} running over the irreducible closed subvarieties $Z\subset X$, and that such an element with strict support $Z$ is a polarizable VHS of weight $w-\dim(Z)$ over some dense open subset of~$Z$.
One important theorem of Saito is that any polarizable VHS $\mathbb{V}$ over $X\subset \overline{X}$ of weight $w$ defines a pure polarizable Hodge module of weight $w+N$ over $\overline{X}$ with full support. The underlying perverse sheaf is the usual intermediate extension functor applied to~$\mathbb{V}$.

There is also a category $\MFW(X)$ whose objects are the objects of $\MF(X)$ equipped with an additional increasing filtration $W_\bullet$ (so $W_\bullet$ is defined on $P$ as well as on $\mathcal{M}$ and $\alpha$ is a filtered quasi-isomorphism). Any VMHS (without additional assumption) over $X$ defines on object of $\MFW(X)$ similarly as in the pure case. Then the category of \emph{mixed Hodge modules} $\MHM(X)$ is a full subcategory of $\MFW(X)$, whose $\Gr^w_W$ of the objects are in $\MH(X,w)$, with additional axioms on the way these graded parts are related (in dimension~$1$ these axioms come from the admissibility condition).

So this is the statement of one of the main theorems of Saito:

\begin{theorem}[{\cite[3.27]{SaitoMHM}}]
Let $\overline{X}$ be a compact complex manifold, let $j:X\hookrightarrow \overline{X}$ be the complement of a normal crossing divisor $D$. Let $\mathbb{V}$ be a VMHS over $X$ admissible in $\overline{X}$. Then this extends to mixed Hodge module over $\overline{X}$ whose underlying perverse sheaf $P$ is $Rj_*\mathbb{V}[N]$ and whose underlying $\mathcal{D}_{\overline{X}}$-module $\mathcal{M}$ is the meromorphic extension of $(\mathbb{V}\otimes\mathcal{O}_X , \nabla)$.
\end{theorem}

In particular the filtration $\mathcal{F}$ of $\mathcal{M}$ is defined using the sheaves $j_* \mathcal{F}$ (the fact that these extend comes from the admissibility condition). The procedure to construct the filtration $\mathcal{W}$ is more complicated and we will need to detail it to show that is behaves well multiplicatively.

But first let us explain how from this we get a MHC computing the cohomology with coefficients in~$\mathbb{V}$ (recall that $\overline{\mathcal{V}}^{\mathopen\lbrack 0, 1\mathclose\rparen}$ is the canonical extension of the vector bundle $\mathcal{V}$ with residues in $\mathopen\lbrack 0, 1\mathclose\rparen$):

\begin{corollary}[{See~\cite[4.12]{Arapura}}]
If $\overline{X}$ is compact Kähler (or proper algebraic), there are filtrations $\mathfrak{W}, F$ on $\Omega^\bullet_{\overline{X}}(\log D,\overline{\mathcal{V}}^{\mathopen\lbrack 0, 1 \mathclose\rparen})$ such that the data
\begin{equation*}
\bigl( (Rj_*\mathbb{V}, \mathfrak{W}),\ (\Omega^\bullet_{\overline{X}}(\log D,\overline{\mathcal{V}}^{\mathopen\lbrack 0, 1 \mathclose\rparen}), \mathfrak{W}, F) \bigr)
\end{equation*}
is a mixed Hodge complex of sheaves over $\overline{X}$.
\end{corollary}

\begin{proof}
We use the inclusion
\begin{equation}
(\Omega^\bullet_{\overline{X}}(\log D,\overline{\mathcal{V}}), \mathfrak{W}_\bullet[N], F_\bullet)[N] \subset (\DR(\mathcal{M}), \mathfrak{W}_\bullet, F_\bullet)
\end{equation}
to define the filtrations of the left-hand side; the shift is necessary because of the convention that a VHS of weight $k$ extends to a Hodge module of weight $k+N$. This becomes a filtered quasi-isomorphism. The weight filtration is constructed in such a way that the terms $\Gr^\mathfrak{W}_k$ of the right-hand side are pure polarizable Hodge modules of weight~$k$.

It remains to check the axioms of MHC for $K:=R\Gamma(\overline{X},-)$, the derived push-forward to a point, and in particular to show that the $\Gr^\mathfrak{W}_k K$ are pure Hodge complexes of weight~$k$. This follows from the theorems of Saito (\cite[Theorem~5.3.1]{SaitoMH}) on push-forwards of pure Hodge modules, applied to the point: the filtration $F$ is strict and the cohomology carries pure Hodge structures. This is written in loc.\ cit.\ for projective morphisms, but on the point this also works for compact Kähler manifolds (see \cite[5.3.8.2]{SaitoMH}), since this relies only on the analytic theory of~\cite{CattaniKaplanSchmid}, and in dimension $1$ this is just the work of Zucker. This also holds over any smooth proper algebraic variety (by Chow's lemma, it is birational to a projective variety) or even a manifold bimeromophic to a compact Kähler manifold (Fujiki class $\mathscr{C}$) in~\cite{SaitoMHM}.
\end{proof}

We now need a concrete understanding of the weight filtration of this MHC and how it behaves multiplicatively. This is easier to do when we assume unipotent monodromy than the general case of quasi-unipotent monodromy.

We work locally and reduce to $\overline{X}$ being a polydisk of dimension $N$ with coordinates $x_1,\dots,x_N$ and $X$ is the complement of a normal crossing divisor~$D$ with components $D_i=\{x_i=0\}$. We assume that the admissible VMHS $\mathbb{V}$ has unipotent monodromy around each component of $D$ and is extended to the $\mathcal{D}_{\overline{X}}$-module $\mathcal{M}$, while the local system $\mathbb{V}$ is extended to a perverse sheaf $P:=Rj_*\mathbb{V}[N]$. The weight filtration is defined in the same way for the $\mathcal{D}_{\overline{X}}$-module and for the perverse sheaf so we focus here on~$\mathcal{M}$.

Along each axis $D_i$ is defined an increasing $V$-filtration $V(i)$ of $\mathcal{M}$ which is indexed by $\mathbb{Z}$ (since we work only with unipotent monodromy). Its properties are that $\Gr_0^{V(i)}\mathcal{M}$ is the vanishing cycles $\psi_{x_i}\mathcal{M}$ of $\mathcal{M}$ for the function $x_i$, and $\Gr_{-1}^{V(i)}\mathcal{M}$ is the nearby cycles $\phi_{x_i}\mathcal{M}$. There are canonical maps $\can_i:\psi_{x_i}\mathcal{M}\rightarrow \phi_{x_i} \mathcal{M}$ given by $-\partial_i$, $\var_i:\phi_{x_i}\mathcal{M}\rightarrow \psi_{x_i} \mathcal{M}$ given by multiplication by $x_i$, and $N_i$ the endomorphism of $\psi_{x_i}\mathcal{M}$ given by the logarithm of the monodromy around $D_i$. 

Now for a subset of indices $I\subset\{1,\dots,N\}$ is defined an object $\Psi_I\mathcal{M}$ by a compositions of $N$ functors $u_1 \circ \cdots \circ u_N$ applied to $\mathcal{M}$, where $u_i$ is either $\psi_{x_i}$ if $i\in I$ or $\phi_{x_i}$ if $i\notin I$; these all commute. The nearby and vanishing cycles induce maps, for $i\not\in I$,
\begin{equation*}
\can_i : \Psi_{I\cup\{i\}}\mathcal{M} \rightleftarrows \Psi_{I}\mathcal{M} : var_i
\end{equation*}
The important fact is that the filtered $\mathcal{D}_{\overline{X}}$-module $\mathcal{M}$ can be entirely reconstructed from the data of such a quiver (\cite{GalligoGrangerMaisonobe}, see also \cite[\S~4.5]{Arapura}, \cite{SaitoAdmissible}). In our case of unipotent monodromy, all $\Psi_I\mathcal{M}$ are vector spaces identified with $\overline{\mathcal{V}}_x$, the canonical fiber at $x=0$. They are all equipped  with a weight filtration~$W_{\bullet}$ coming directly from the structure of VMHS, and commuting nilpotent endomorphisms $N_i$ compatible with~$W_\bullet$.

So, the appropriate weight filtration $\mathfrak{W}$ of the mixed Hodge module can also be understood from this quiver. It is obtained by replacing $W_{\bullet}$ on $\Psi_I\mathcal{M}$ by another filtration $W_{I,\bullet}$. If $I = \{i\}$ then $W_{I, k}$ is given, as in the case of curves, by the filtration
$Z_i(W)_{k} : =N_i(W_k) + M_{k-1} \cap W_{k-1}$ where $M$ is the relative monodromy filtration for $N_i$. Then this process has to be iterated: if $I = \{i_1, \dots, i_r\}$ then
\begin{equation*}
W_{I,k} := Z_{i_1}(Z_{i_2}(\cdots Z_{i_r}(W)))_k .
\end{equation*}
The $W_{I,\bullet}$ all together define a filtration of the quiver and then a filtration $\mathfrak{W}_\bullet$ of the $\mathcal{D}_{\overline{X}}$-module $\mathcal{M}$ and of the perverse sheaf $P$.

\begin{theorem}
Under the previous hypotheses and assuming that the monodromy of $\mathbb{V}$ is unipotent around~$D$, there are filtrations $\mathfrak{W},F$ such that
\begin{equation*}
\bigl( (R_\TW j_*\mathbb{V}, \mathfrak{W}),\ (\Omega^\bullet_{\overline{X}}(\log D,\overline{\mathcal{V}}), \mathfrak{W}, F) \bigr)
\end{equation*}
is a lax symmetric monoidal MHC of sheaves.
\end{theorem}

\begin{proof}
As previously, there is a chain of (multiplicative, but not yet filtered) quasi-isomorphisms relating these two parts, and it is lax symmetric monoidal when we work with unipotent monodromy.

The filtration $\mathfrak{W}$ is put on $R_\TW j_*\mathbb{V}$ simply because this is isomorphic as perverse sheaf to the previous~$P$, hence the elements $\mathfrak{W}_k \subset P$ are sub-perverse sheaves and correspond to sub-objects of $R_\TW j_*\mathbb{V}$, up to this shift by the dimension of $\overline{X}$.

Then it remains only to check that the construction of the filtration $\mathfrak{W}$ behaves well multiplicatively and we see that this is a generalization of the case of curves. Given the two VMHS $\mathbb{V}_1, \mathbb{V}_2$ on $X$, admissible and unipotent, and extended to mixed Hodge modules on $\overline{X}$ with underlying $\mathcal{D}_{\overline{X}}$-module $\mathcal{M}_1,\mathcal{M}_2$, the weight filtration is described as above with the filtered quiver formed by the $\Psi_I\mathcal{M}_1$ and $\Psi_I\mathcal{M}_2$.  The tensor product VMHS $\mathbb{V}=\mathbb{V}_1 \otimes \mathbb{V}_2$ with its tensor product weight filtration $W_\bullet$ is extended to a mixed Hodge module with underlying $\mathcal{D}_{\overline{X}}$-module $\mathcal{M}$ which is exactly $\mathcal{M}_1 \otimes \mathcal{M}_2$ when forgetting the filtrations.

The important fact is that in our previous description of $\mathcal{D}_{\overline{X}}$-modules near a point of normal crossings, then under tensor product, $\Psi_I\mathcal{M}_1$ and $\Psi_J\mathcal{M}_2$ get paired to $\Psi_{I\cup J}\mathcal{M}$ if $I\cap J=\emptyset$ or $0$ else. This is because each $\Psi_I$ corresponds to applying $\Gr^{V(i)}_{-1}$ for $i\in I$ and $\Gr^{V(i)}_{0}$ for $i\notin I$.

So the compatibility to check between the tensor products of filtrations involves similar computations as in the proof of Theorem~\ref{theorem:SteenbrinkZucker-multiplicative} (the case of curves), that is,
\begin{equation*}
W^1_k \otimes W^2_\ell \subset W_{k+\ell}
\end{equation*}
and
\begin{equation*}
W^1_k \otimes Z_i(W^2)_\ell \subset Z_i(W)_{k+\ell}
\end{equation*}
applied several times
\end{proof}

\section{Application to the deformation theory of representations of fundamental groups}
\label{section:applications}

In this final section we use the constructions of mixed Hodge diagrams of Lie algebras to describe the formal deformation theory of representation of the fundamental group of a variety $X$. Thus the results presented here extend and improve directly the results of \cite{Lefevre2}.

\subsection{Representations of fundamental groups}

Let $X$ be any differentiable manifold with fundamental group $\pi_1(X,x)$ that is finitely presentable. Let $G$ be a linear algebraic group defined over the field $\kk\subset\RR$ or $\kk=\CC$ with Lie algebra $\mathfrak{g}$. The representations
\begin{equation*}
\rho:\pi_1(X,x)\rightarrow G(\kk)
\end{equation*}
are parametrized by the $\kk$-points of an affine scheme of finite type $\Hom(\pi_1(X,x), G)$ (see \cite{LubotzkyMagid}). Fixing such a representation $\rho$, let $\Ohat_\rho$ be the completion of the local ring of $\Hom(\pi_1(X,x), G)$ at its $\kk$-point $\rho$. It has an associated deformation functor $\Def_\rho$, which is a functor from the category of local Artin $\kk$-algebras with residue field $\kk$ (simply denoted by $\Art$) to the category of sets given by
\begin{equation}
\Def_\rho : A\longmapsto\bigl\{ \tilde{\rho}:\pi_1(X,x)\rightarrow G(A) \bigm|
\tilde{\rho}=\rho\ \text{over}\ G(\kk) \bigr\} = \Hom(\Ohat_\rho, A).
\end{equation}

Over $\CC$, the representation $\rho$ defines a flat principal $G$-bundle $P$ over $X$. The associated bundle
\begin{equation}
\ad_\rho := \ad(P) = P \times^G \mathfrak{g}
\end{equation}
where $G$ acts on $\mathfrak{g}$ via $\Ad$ is then a flat vector bundle with a Lie bracket. So the algebra $L$ of $\mathcal{C}^\infty$ differential forms over $X$ with values in $\ad_\rho$ has the structure of a differential graded Lie algebra: the bracket combines the wedge product of differential forms and the Lie bracket in $\mathfrak{g}$, and the differential acts only on differential forms since $\ad_\rho$ is flat.

To any such DG Lie algebra (over any field $\kk$ of characteristic zero) is associated a deformation functor $\Def_L$. For $A\in\Art$ with its unique maximal ideal $\mathfrak{m}_A$, then $L^0\otimes\mathfrak{m}_A$ is a \emph{nilpotent} Lie algebra, thus has a group structure denoted by
\begin{equation*}
(\exp(L^0\otimes\mathfrak{m}_A),*)
\end{equation*}
given by the Baker-Campbell-Hausdorff formula, and acts on $L^1 \otimes\mathfrak{m}_A$ by \emph{gauge transformations}. Then $\Def_L$ is given by
\begin{equation}
\Def_L(A, \mathfrak{m}_A):=
\biggl\{ \omega\in L^1\otimes\mathfrak{m}_A \biggm|
0=d(\omega)+\frac{1}{2}[\omega,\omega] \in L^2\otimes\mathfrak{m}_A \biggr\}/\exp(L^0 \otimes\mathfrak{m}_A)
\end{equation}
The important theorem is that this is invariant of $L$ under quasi-isomorphisms.

In our case $L$ is equipped with an augmentation $\epsilon:L\rightarrow\mathfrak{g}$ which evaluates differential forms at $x$. To this is associated an \emph{augmented deformation functor} $\Def_{L,\epsilon}$ which is a small variation of $\Def_L$, explicitly written by Eyssidieux-Simpson \cite[\S~2.1.1]{EyssidieuxSimpson}:
\begin{equation}
\Def_{L,\epsilon}(A, \mathfrak{m}_A) := \\
\biggl\{ (\omega, e^\alpha) \in (L^1\otimes\mathfrak{m}_A) \times \exp(\mathfrak{g}\otimes\mathfrak{m}_A) \biggm|
0=d(\omega)+\frac{1}{2}[\omega,\omega] \biggr\} / \exp(L^0 \otimes\mathfrak{m}_A)
\end{equation}
where, for $e^\lambda \in \exp(L^0 \otimes\mathfrak{m}_A)$,
\begin{equation}
e^\lambda\cdot(\omega,e^\alpha) := (e^\lambda \cdot\omega,\ e^\alpha*e^{-\epsilon(\lambda)}) .
\end{equation}
See \cite{ManettiLectures} for many more details on formal deformation theory and DG Lie algebras.

Part of the main theorem of Goldman-Millson can be stated as follows:

\begin{theorem}[{Goldman-Millson \cite{GoldmanMillson}}]
For any differentiable manifold $X$ with finitely presentable fundamental group, and for a representation $\rho$ of $\pi_1(X,x)$ into a linear algebraic group $G$, with $L$ the DG Lie algebra of $\mathcal{C}^\infty$ differential forms with values in $\ad(\rho)$, there is an isomorphism of deformation functors
\begin{equation*}
\Def_\rho \simeq \Def_{L,\epsilon} .
\end{equation*}
Furthermore, a quasi-isomorphism of DG Lie algebras $L\rightarrow M$ augmented over $\mathfrak{g}$ induces an isomorphism of deformation functors $\Def_{L,\epsilon}\rightarrow \Def_{M,\epsilon}$.
\end{theorem}

In other words $\Ohat_\rho$ can be computed from the data of $L$ above $\mathfrak{g}$ up to quasi-isomorphism. In our previous work, we found canonical formulas compatible with the notion of mixed Hodge complex to construct a functorial MHS on $\Ohat_\rho$ in situations where $L$ has a structure of mixed Hodge diagram of Lie algebras. Implicitly, we consider $\Ohat_\rho$ to be a projective limit
\begin{equation*}
\Ohat_\rho = \lim_{\leftarrow} \Ohat_\rho/\mathfrak{m}^n
\end{equation*}
of its quotients by powers of its maximal ideal, that are \emph{finite-dimensional}, hence a MHS on $\Ohat_\rho$ means a projective limit of MHS's, that are also compatible with their structure of algebra.

\begin{theorem}[{\cite{Lefevre2}}]
\label{theorem:main-construction-previous-work}
When $\rho$ is a representation of $\pi_1(X,x)$ for which one can construct a mixed Hodge diagram of Lie algebras $L$ over $\kk$ computing the cohomology of $X$ with coefficients in $\ad_\rho$, with its augmentation $\epsilon:L\rightarrow\mathfrak{g}$ at $x$, then one can construct a MHS on $\Ohat_\rho$ over $\kk$. It is functorial in $L$ (over $\mathfrak{g}$) and invariant under quasi-isomorphisms.
\end{theorem}

\begin{remark}
The proof in loc.\ cit.\ requires a somewhat unnatural hypothesis of non-negative weights on $\mathfrak{g}$ (\cite[Theorem~8.4]{Lefevre2}) which is satisfied here only for VHS. It is needed in order to show that the mapping cone $C$ of $\epsilon$ is a mixed Hodge diagram of $L_\infty$ algebras for the definition given there of those objects; the second step is to perform the \emph{bar construction} on $C$ which is a DG coalgebra and show that this is again a MHC. However:
\begin{enumerate}
\item Even without hypothesis on the weights of $\mathfrak{g}$, one can still show that the bar construction is a MHC, without claiming that the $L_\infty$ structure of $C$ (see below) is compatible with the weight filtration,
\item In future work \cite{Lefevre4} we will give another proof of this which holds without such hypotheses,
\item The rest of the present work doesn't really depend on the existence of a MHS on $\Ohat_\rho$, but only on the data of $L$ and its augmentation.
\end{enumerate}
\end{remark}

Combining this with the constructions through \S~\ref{section:non-degenerate}--\ref{section:MHM}, we deduce:

\begin{theorem}
\label{theorem:MHS-functorial}
Let $\rho$ be the monodromy representation of an admissible VMHS $\mathbb{V}$ over $\kk$ over a smooth algebraic variety or a quasi-Kähler manifold $X$ with unipotent monodromy at infinity. Then there is a MHS over $\kk$ on $\Ohat_\rho$ that is functorial in the data of $X, x, \rho$.
\end{theorem}

In the algebraic case, we have to prove the independance on the compactification.
In the quasi-Kähler case, we assume that part of the data is an equivalence class of compactifications of $X$ where $\mathbb{V}$ is admissible and has unipotent monodromy at infinity, where two compactifications are equivalent if they are dominated by a third one; and morphisms are required to extend to the compactifications. In both cases, the condition of unipotent monodromy at infinity is independent on the compactification.

\begin{proof}[Proof of Theorem~\ref{theorem:MHS-functorial}]
Consider a compactification $X\subset\overline{X}$ by a divisor with normal crossings, and such that $\mathbb{V}$ has unipotent monodromy around $D$.

If $\rho$ is the monodromy of the VMHS $\mathbb{V}$ over $X$, then $\ad_\rho\subset\End(\mathbb{V})$ is again the monodromy of a VMHS by linear algebra. It is again admissible if $\rho$ is (\cite[Appendix]{SteenbrinkZucker}) and has again unipotent monodromy.

Then by the constructions of this article we get the MHC of sheaves $\operatorname{\mathcal{MHC}}(X,\overline{X},\ad_\rho)$ and then $\operatorname{MHC}(X,\overline{X},\ad_\rho)$ by applying $R_\TW(\overline{X},-)$, and this is a mixed Hodge diagram of Lie algebras, as explained in \S~\ref{subsection:application-Hodge} and in \S~\ref{subsection:results}. It is defined over the same field~$\kk$ as~$\mathbb{V}$. We also need the augmentation $\epsilon$ to be a morphism of mixed Hodge diagrams, when $\mathfrak{g}$ ($=(\ad_\rho)_x$) is considered as a mixed Hodge diagram concentrated in degree zero. This is obvious, indeed the augmentation is the morphism of MHC induced by the map $x:*\rightarrow X$. Hence we can apply Theorem~\ref{theorem:main-construction-previous-work} and get a MHS over $\kk$ on~$\Ohat_\rho$.

This MHS is functorial in $\mathbb{V}$ and $x\in X$ on one side, and on the data $(X,\overline{X})$ on another side. But by usual arguments this implies the independence on the compactification: a morphism
\begin{equation*}
f : (X, \overline{X}^2) \longrightarrow (X, \overline{X}^1)
\end{equation*}
induces a quasi-isomorphism
\begin{equation*}
\operatorname{MHC}(X, \overline{X}^2, \ad_\rho) \longrightarrow \operatorname{MHC}(X, \overline{X}^1, \ad_\rho)
\end{equation*}
so the MHS's that these define on $\Ohat_{\rho}$ are the same; and two compactifications with local unipotent monodromy are dominated by a third one (in the algebraic case this is a theorem, in the quasi-Kähler case this is included as part of the data).

Let us also explain why the local unipotent monodromy hypothesis is independent on the compactifications. It can be reformulated as : every point of $\overline{X}$ has a basis of neighborhoods $U$ such that the Zariski closure of the image of $\pi_1 (X \cap U)$ in the monodromy group of $\mathbb{V}$ is an abelian unipotent algebraic subgroup. Given the two compactifications $\overline{X}^1$ and $\overline{X}^2$, we can then show that $\mathbb{V}$ has unipotent monodromy around $D_1 := \overline{X}^1 \setminus X$ if and only if it has unipotent monodromy around $D_2 := \overline{X}^1 \setminus X$. Going from $\overline{X}^1$ to $\overline{X}^2$ is clear by pulling back; conversely, if $\mathbb{V}$ has unipotent monodromy around $D_2$ then the monodromy around $D_1$ is generated by small commuting loops that are images by $f$ of loops around $D_2$ that all generate an abelian unipotent subgroup, so these all together generate and abelian unipotent subgroup.

And then, since morphisms extend to compactifications, this implies the functoriality in $X,x$ only.
\end{proof}

\begin{remark}
In the cases of our \S~\ref{section:non-degenerate} and \S~\ref{section:curves} (non-degenerate VMHS or curves, not using mixed Hodge modules) we can also talk about MHS and MHC over the field $\CC$ of Remark~\ref{remark:C-MHS} and the above Theorem is still true over $\CC$. When using mixed Hodge modules, it seems that this also works as explained in~\cite[\S~3.2]{SabbahDettweiler}.
\end{remark}

\subsection{Deformation functor of DG Lie and $L_\infty$ algebras}
We are going to describe the consequence of $L$ being a mixed Hodge diagram of Lie algebras, with restrictions on the possible weights on cohomology, to its deformation functor.

To have a better understanding of the deformation functor of DG Lie algebras, we will work with the much more powerful $L_\infty$ algebras. Briefly, a $L_\infty$ algebra is given by a graded vector space $L$ with anti-symmetric operations in $r$ variables of degree $r-2$
\begin{equation*}
\ell_r : L^{\otimes r} \longrightarrow r, \quad r\geq 1
\end{equation*}
satisfying an infinite list of axioms. Among these, $\ell_1$ is a differential $d$ on $L$ and $\ell_2$ behaves likes a Lie bracket, for which $d$ is a derivation, except that it satisfies the Jacobi identity only \emph{up to homotopy} given by $\ell_3$, and $\ell_3$ itself satisfies higher order relations. $L_\infty$ algebras enjoy the following very nice properties:
\begin{enumerate}
\item $L_\infty$ algebras with $\ell_r=0$ for all $r\geq 3$ are the same as DG Lie algebras.
\item \label{item:transfer-structure-Linf} If $L$ is a DG Lie algebra then its cohomology $H(L)$ comes equipped with a structure of $L_\infty$ algebra such that $L$ and $H(L)$ become quasi-isomorphic in the sense of $L_\infty$ algebras. This is called a \emph{homotopy transfer of structure} from $L$ to $H(L)$.
\item As a consequence, a DG Lie algebra $L$ is \emph{formal} (i.e.\ quasi-isomorphic to its cohomology) if and only if there exists a homotopy transfer of structure to $H(L)$ with $\ell_r=0$ for all $r\neq 2$ ($\ell_1$ is always zero on $H(L)$ since it is induced by the differential).
\end{enumerate}
See the lectures of Manetti \cite{ManettiLectures} and the book of Loday-Vallette \cite{LodayVallette} for much more motivation for $L_\infty$ algebras.

Such a $L$ also has a deformation functor on $\Art$ which is given by
\begin{equation}
\Def_L:(A,\mathfrak{m}_A) \longmapsto \Biggl\{ \omega\in L^1\otimes\mathfrak{m}_A \Biggm|
0=\sum_{r\geq 1} \frac{\ell_r(\omega,\dots,\omega)}{r!} \in L^2\otimes\mathfrak{m}_A \Biggr\} / \sim
\end{equation}
where $\sim$ is the homotopy equivalence relation: $\omega_0 \sim \omega_1$ if there exists an element
\begin{equation*}
\gamma\in \mathfrak{m}_A \otimes \kk[t,dt]
\end{equation*}
(where $\kk[t,dt]$ is the DG algebra of polynomial $1$-form on the affine line) that satisfies the above equation in $L\otimes \mathfrak{m}_A \otimes \kk[t,dt]$ and such that $\gamma(t=0)=\omega_0$ and $\gamma(t=1)=\omega_1$. Again it is invariant under quasi-isomorphisms. If $\ell_r=0$ for $r\geq 3$ we recover the previous deformation functor.

\subsection{Homotopy transfer of structure}
\label{homotopy-transfer-structure}
We need to recall briefly how the homotopy transfer of structure works.

Let $L$ be an $L_\infty$ algebra with operations denoted by $\mu_r$. This includes the case of DG Lie algebras with $\mu_r=0$ for $r>2$. What we need to choose is a decomposition in each degree
\begin{equation}
L^n := A^n \oplus K^n \oplus B^n
\end{equation}
where $A^n$ is a complement to $\Ker(d^n)\subset L^n$ and $B^n$ is a complement to $\Im(d^{n-1})\subset\Ker(d^n)$. Then $K^n$ forms a space of representatives for the cohomology $H^n(L)$. This also determines maps of complexes $i:K^\bullet\rightarrow L^\bullet$ (inclusion), $p:L^\bullet\rightarrow K^\bullet$ (projection), and $h:L^\bullet\rightarrow L^{\bullet-1}$ (homotopy) where $h$ is given by the inverse of the isomorphism $A^{\bullet-1}\rightarrow B^\bullet$ induced by $d$, extended by $0$ to $A\oplus K$. These maps satisfy $pi=\id_K$ and $\id_L-ip=dh-hd$, so $L$ deformation retracts onto $K$.

Once such a choice is a made, it determines operations $\ell_r$ on $K\simeq H(L)$ with $\ell_1=0$ and such that $L$ and $H(L)$ become quasi-isomorphic in the sense of $L_\infty$ algebras via $i$ and $p$.

One way to describe $\ell_r$ is via the set $RT_r$ of rooted trees with $r$ leaves, that is, trees that when written vertically, have $r$ leaves thought of as input data and one leave as output data, with internal vertices also presented vertically with at least two input edges and one output edge. Such a tree $T$ corresponds to a composition scheme for a sequence $(\mu_n)$ of operations in $n$ variables, by plugging one operation $\mu_n$ in an internal vertex with $n$ inputs. One can also label the edges by linear maps. The formula for $\ell_r$ (see \cite[\S~10.3.4]{LodayVallette}) is then the sum over all trees $T\in RT_r$ of the following operations in $r$ variables: apply the maps $i$ on the input leaves of $T$, maps $\mu_n$ on the internal vertices with $n$ inputs, maps $h$ on the internal edges, and a map $p$ on the last output edge.

\subsection{Higher operations and weights}
Let us explain how powerful this point of view is. Because of the homotopy transfer of structure and the invariance by quasi-isomorphism of the deformation functor, then the deformation functor of any DG Lie algebra $L$ can be written directly in $H(L)$. We will work in cases where $L$ has $H^n(L)=0$ for $n\leq 0$, so the homotopy relation in the formula for $\Def_{H(L)}$ is trivial, and the other terms $H^n(L)$ are finite-dimensional. Hence the formula for $\Def_{H(L)}$ gives us directly a complete local algebra pro-representing $\Def_L$: it is given as the quotient of the power series on $H^1(L)$ by the equations $0=\sum_{r\geq 2} \ell_r(x,\dots,x)/r!$ seen as power series with values in $H^2(L)$.

If furthermore one can show that only finitely many of the operations $\ell_r:H^1(L)^{\otimes r}\rightarrow H^2(L)$ are non-vanishing, then the formula for $\Def_{H(L)}$ gives a finite presentation of this complete local algebra: it is the completion of the local ring of the germ at $0$ inside $H^1(L)$ defined by a finite number of polynomial equations. For brevity we will simply say that the deformation functor is given by a finite number of polynomial equation.

Now we will use this combined with the existence of weights on $H(L)$.

\begin{theorem}
\label{theorem:weight-grading}
Let $L$ be a mixed Hodge diagram of Lie algebras over $\kk$. Then on $H(L)$ there is an extra grading $H_\bullet(L)$ over $\kk$ that splits the weight filtration, i.e.\
\begin{equation*}
W_k H^n(L)=\bigoplus_{i\leq k} H^n_i(L) ,
\end{equation*}
and there are induced higher operations $\ell_r$ that all respect this grading, i.e.\
\begin{equation*}
\ell_r(H^{n_1}_{k_1}(L),\dots, H^{n_r}_{k_r}(L)) \subset H^{n_1+\cdots+n_r-2}_{k_1+\cdots+k_r}(L) .
\end{equation*}
\end{theorem}

\begin{proof}
The first essential step is to use the theorem of Cirici-Horel \cite[Theorem~7.8, Theorem~8.2]{CiriciHorel} which shows that $(L,W)$, over $\kk$, is quasi-isomorphic to $\mathscr{M}:=E_1^W(L)$ (the first page of the $W$-spectral sequence). Since their theorem is formulated as a formality result for mixed Hodge complexes as a symmetric monoidal category, it holds as well for our DG Lie algebras. Since the $W$-spectral sequence degenerates at $E_2$, $H(\mathscr{M})=H(L)$, and this $\mathscr{M}$ already comes equipped with a weight grading $\mathscr{M}_\bullet$ that on cohomology splits the weight filtration of $H(L)$.

The second essential step is to show that the theorem on homotopy transfer of structure holds with an extra grading, see \cite[Corollary~5.6]{BudurRubio} for the case of commutative DG algebras. But it is clear that this also works for $L_\infty$ algebra since the splitting of $\mathscr{M}$ can be taken in a compatible way with its grading and then the maps constructed from the homotopy transfer of structure will respect this grading. So we get higher operations on $H(\mathscr{M})$ that respect the grading.
\end{proof}

Now this game of higher operations and weights allows us to show that many operations vanish automatically when they don't respect the weights, and this leads to a quite concrete description of $\Def_L$ in many cases.

\begin{theorem}
\label{theorem:weights-finite-presentation}
Assume that $L$ is a mixed Hodge diagram of Lie algebras with $H^n(L)=0$ for $n\leq 0$ and such that the weights of $H^1(L)$ are all strictly positive. Then the complete local algebra that pro-represents $\Def_L$ has a finite presentation with weights on the generators being the weights of $H^1(L)$ and with weighted-homogeneous relations with weights those of $H^2(L)$.
\end{theorem}

\begin{proof}
Because of the previous theorem, we compute the deformation functor in $H(L)$. The main point is that if $x_1,\dots,x_r \in H^1(L)$ have weight $k_1,\dots,k_r>0$ then $\ell_r(x_1,\dots,x_r)\in H^2(L)$ has weight $k_1+\cdots+k_r\geq r$. But $H^2(L)$, being finite-dimensional, has only finitely many weights. Hence for some $N$ big enough all the operations $\ell_r:H^1(L)^{\otimes r}\rightarrow H^2(L)$ vanish for $r>N$. The equations giving the deformation functor
\begin{equation*}
\frac{\ell_2(x,x)}{2}+\frac{\ell_3(x,x,x)}{3!}+\cdots+\frac{\ell_N(x,\dots,x)}{N!}=0
\end{equation*}
are really a finite number of polynomial equations with weights on the variable $x$ in $H^1(L)$ and the relations take values in $H^2(L)$ respecting those weights.
\end{proof}

\subsection{The augmentation}

To apply the above method with the isomorphism of functors of Goldman-Millson, we need to work with the augmented deformation functor $\Def_{L,\epsilon}$. In \cite[\S~5]{Lefevre2} we argue that it is actually the deformation functor of a canonical $L_\infty$ algebra structure on the (suspended) mapping cone $C$ of $\epsilon$ that is constructed and studied by Fiorenza-Manetti (\cite{FiorenzaManetti}). We will denote by $\mu_r$ ($r\geq 1$) these operations.

We write these formula directly in the form we need: $C$ has components
\begin{equation}
C^1=L^1\oplus\mathfrak{g}, \quad C^n=L^n \ (n\neq 1).
\end{equation}
The differential $\mu_1$ is induced by the one of $L$ with the exception of $d^1:C^0\rightarrow C^1$ given by
\begin{equation}
d^1(x) := (d(x), \epsilon(x)), \quad x\in L^0 .
\end{equation}
The operation $\mu_2$ is induced in a natural way by the Lie bracket, the only part needed to be described being $C^0\otimes C^1\rightarrow C^1$ given by
\begin{equation}
\mu_2(x, (y,v)) := \left([x,y], \frac{1}{2}\,[\epsilon(x),v]\right), \quad x\in L^0,y\in L^1, v\in\mathfrak{g} .
\end{equation}
Furthermore, there is one operation $\mu_r$ for $r\geq 3$ as follows. Applied to $r$ elements of $\mathfrak{g}$ and $k>1$ elements of $L$, $\mu_{r+k}$ is $0$. Else, applied to $r$ elements $u_1,\dots,u_r$ of $\mathfrak{g}$ and exactly one element $x$ of $L$ it is given by a formula of the type
\begin{equation}
\mu_{r+1}(u_1,\dots,u_r,x)=B_r\sum_\tau \pm [u_{\tau(1)}, [u_{\tau(2)},\dots,
[u_{\tau(r)}, \epsilon(x)] \dots ]]
\end{equation}
where the sum is over the symmetric group, with a constant $B_r$ and with signs in the sum. Since this has values in $\mathfrak{g}$, $\mu_{r+1}$ is a non-trivial higher operation $(C^1)^{\otimes r}\otimes C^0 \rightarrow C^1$ for all $r\geq 2$.

Then:
\begin{lemma}[{\cite[5.3]{Lefevre2}}]
For this $L_\infty$ algebra structure on $C$,
\begin{equation*}
\Def_{L,\epsilon} := \Def_C  .
\end{equation*}
\end{lemma}

We are in situations where $\epsilon$ is surjective and the induced $\epsilon$ on $H^0(L)$ is injective. Hence $H^0(C)=0$. Let us recall also that there is a long exact sequence for the mapping cone, which reduces here to the short exact sequence
\begin{equation}
0\longrightarrow \mathfrak{g}/\epsilon(H^0(L)) \longrightarrow H^1(C) \longrightarrow
H^1(L) \longrightarrow 0
\end{equation}
and of course $H^n(C)=H^n(L)$ for $n\neq 0,1$. If $L$ is an augmented mixed Hodge diagram of Lie algebras, then the above sequence is a short exact sequence of MHS.

With all of this, we are ready to state our theorem, first without Hodge theory.

\begin{theorem}
Let $\epsilon:L\rightarrow\mathfrak{g}$ be an augmented DG Lie algebra, let $C$ be the cone as above with its structure of $L_\infty$ algebra, assuming that $\epsilon$ is surjective on $L^0$ and that the induced $\epsilon$ on $H^0(L)$ is injective. Then there exist a transferred structure of $L_\infty$ algebra on $H(C)$ given by operations $\ell'_r$, and a transferred structure of $L_\infty$ algebra on $H(L)$ given by operations $\ell_r$, such that the deformation functor of $H(C)$ can be written as a product
\begin{equation*}
\Def_{H(C)}=\Def_{H(L)}\times (\mathfrak{g}/\epsilon(H^0(L))),
\end{equation*}
in other words for $(A,\mathfrak{m}_A)\in \Art$
\begin{equation*}
\Def_{H(C)}(A)=
\Biggl\{ (x,t)\in \big(H^1(L)\oplus (\mathfrak{g}/\epsilon(H^0(L)))\big)\otimes \mathfrak{m}_A \Biggm|
0=\sum_{r\geq 2} \frac{\ell_r(x,\dots,x)}{r!} \Biggr\} .
\end{equation*}
\end{theorem}

\begin{proof}
For this we prepare a homotopy transfer of structure for $L$, so we choose a decomposition
\begin{equation}
L^n := A^n \oplus K^n \oplus B^n
\end{equation}
where $K$ forms a space of representatives for the cohomology. This determines again the maps $i,p,h$ and higher operations $\ell_r$ on $K\simeq H(L)$.

Then we choose such a decomposition for $C$ simply by choosing a subspace $\mathfrak{t}\subset\mathfrak{g}$ complement to $\epsilon(H^0(L))$. This gives a splitting of $C$ with
\begin{equation}
C^1 := A^1 \oplus (K^1\oplus\mathfrak{t}) \oplus (B^1 \oplus\epsilon(H^0(L))) .
\end{equation}
and determines again maps $i',p',h'$ (closely related to $i,p,h$) and higher operations $\ell'_r$ in $H(C)$ with
\begin{equation}
H^1(C)\simeq K^1\oplus\mathfrak{t} \simeq H^1(L)\oplus (\mathfrak{g}/\epsilon(H^0(L))) .
\end{equation}

What we then want to show is that the only non-zero operations $\ell'_r:H^1(C)^{\otimes r}\rightarrow H^2(C)=H^2(L)$ are the one induced from $\ell_r$, i.e.\ that
\begin{equation}
\text{for}\ y=(x,t)\in K^1 \oplus\mathfrak{t},\quad \ell'_r(y,\dots,y)=\ell_r(x,\dots,x).
\end{equation}
This gives directly the above form of the deformation functor.
Such a fact is already clear for $r=1$ ($\ell'_1$ is always zero on cohomology) and for $r=2$ since $\ell'_2$ is induced by $\mu_2$ on cohomology and $\mu_2$ satisfies this.

For higher $r$ we have to understand in more detail what happens in the homotopy transfer of structure for $C$. We take $r$ elements $y_j=x_j+t_j\in K^1\oplus\mathfrak{t}$ and want to compute $\ell'(y_1,\dots,y_r)$. Let $T$ be a planar rooted tree as in \S~\ref{homotopy-transfer-structure} with $r$ leaves. We first apply apply a map $i'$ to some $n$ elements $y_{j_1},\dots,y_{j_n}$ then apply $\mu_n$. If $n=2$ then what happens is as before, $\mu_2(i(y_{j_1}),i(y_{j_2}))$ equals $\mu_2(i(x_{j_1}),i(x_{j_2}))$ by definition of $\mu_2$. If $n>2$ then $\mu_n(i(y_{j_1}),\dots,i(y_{j_n}))$ is zero simply because there is no such higher operation on $C$ (the elements $y_{j_1},\dots,y_{j_n}$ are all of degree $1$). Hence we see that the formula for computing $\ell'(y_1,\dots,y_r)$ is the same as the one for computing $\ell(y_1,\dots,y_r)$.
\end{proof}

\begin{remark}
If $L$ is formal above $\mathfrak{g}$, then the above computations show that $\Def_L$ reduces to the product of the equation $[x,x]=0$ in $H^1(L)$ and the vector space $\mathfrak{g}/\epsilon(H^0(L))$. Thus we recover completely the result of Goldman-Millson \cite[Theorem~3.5]{GoldmanMillson}, purely by methods of $L_\infty$ algebras and without invoking the operation that they denote by $\bowtie$ (homotopy fiber product of groupoids \cite[\S~2.1.1]{EyssidieuxSimpson} used to define $\Def_{L,\epsilon}$ from $\Def_L$).
\end{remark}

In geometric situations, this is what we get.

\begin{corollary}
Let $\epsilon:L\rightarrow\mathfrak{g}$ be an augmented mixed Hodge diagram of Lie algebras, with $H^n(L)=0$ for $n<0$ and
$\epsilon$ injective on $H^0(L)$. Assume that all the weights of $H^1(L)$ are strictly positive. Then again $\Def_{L,\epsilon}$ has a finite presentation, which is the product of the one of Theorem~\ref{theorem:weights-finite-presentation} with the vector space $\mathfrak{g}/\epsilon(H^0(L))$.
\end{corollary}

\begin{proof}
First we apply, as in the proof of Theorem~\ref{theorem:weight-grading}, the theorem of Cirici-Horel to the map $\epsilon:L\rightarrow\mathfrak{g}$ over~$\kk$. This map is quasi-isomorphic to an augmented DG Lie algebra $\tau:\mathscr{M}\rightarrow\mathfrak{h}$ with an extra grading (on $\mathscr{M}$ and on $\mathfrak{h}$, respected by $\tau)$ splitting the weight filtration at the level of cohomology. Hence $\mathfrak{h}$ is just a splitting of the weight filtration of $\mathfrak{g}$. This also defines a weight grading of the cone $\mathscr{C}$ of $\tau$, which is quasi-isomorphic to $C$, compatible with the operations $\mu_r$ of the cone, and splitting the weight filtration on $H(C)$.

Then we combine the method of Theorem~\ref{theorem:weight-grading} with the previous theorem: we choose a splitting for $\mathscr{C}$ by combining a splitting for $\mathscr{M}$ and a splitting for $\tau(H^0(\mathscr{M}))\subset\mathfrak{h}$, both in a compatible way with the grading. The homotopy transfer of structure with an extra grading then gives us operations $\ell'_r$ on $H(C)$ respecting the weight grading, and at the level of $H^1(C)^{\otimes r}\rightarrow H^2(L)$ these coincide with the operations $\ell_r$ of Theorem~\ref{theorem:weights-finite-presentation}.
\end{proof}

\subsection{Consequences}
\label{subsection:consequences}

Finally, we come back to the study of the complete local ring $\Ohat_\rho$ of the representation variety of $\pi_1(X,x)$ at a representation $\rho$ for various cases for $X$ and $\rho$ where one can construct the augmented mixed Hodge diagram of Lie algebras $L$. The previous description of $\Def_{H(C)}=\Def_{L,\epsilon}$ gives us directly a description of $\Ohat_\rho$ with its weight grading at least in the cases where there are only a finite number of non-vanishing higher operations $\ell_r$ on $H(L)$.

Recall that when $\rho$ is the monodromy representation of a VHS $\mathbb{V}$ then $\ad_\rho\subset\End(\mathbb{V})$ is a VHS of weight zero, hence in all our cases $H^i(X,\ad_\rho)$ has weights greater or equal to $i$ (from a higher-level point of view: the functor $R^if_*$ increases the weights by $i$).

\begin{enumerate}
\item If $X$ is compact Kähler and $\rho$ is the monodromy of a polarized VHS over $X$, then $H^1(L)$ is pure of weight $1$ and $H^2(L)$ is pure of weight $2$. We already know, or we recover, that $L$ is formal and that there are no operations $\ell_r$ for $r\geq 3$ and $\ell_2$ is the Lie bracket induced on cohomology. We recover well again the result of Goldman-Millson: $\Ohat_\rho$ is given by the product of the equation $[x,x]=0$ in $H^1(L)$ (quadratic) with the vector space $\mathfrak{g}/\epsilon(H^0(L))$.

\item If $X$ is smooth algebraic and $\rho$ has finite image then we constructed in \cite[\S~11]{Lefevre1} the mixed Hodge diagram by using a finite cover $Y\rightarrow X$ over which $\rho$ is trivial and constructing the diagram equivariantly on $Y$. This also works in the quasi-Kähler case because a finite cover of a quasi-Kähler manifold is known to be quasi-Kähler (\cite[Proposition~A.5]{ArapuraNori}), with the equivariant compactification and equivariant resolution of singularities still available in this context. Then the weights of $L$ are directly related to the weights $Y$. Hence $H^1(L)$ has weights only $1,2$ and $H^2(L)$ has weights only $2,3,4$. The non-vanishing operations $\ell_r$ exist only for $r=2,3,4$: only $\ell_2(x,x)$ for $x$ of weight $1,2$ and $\ell_3(x,x,x)$, $\ell_4(x,x,x,x)$ for $x$ of weight $1$ produce weights allowed in $H^2(L)$. Furthermore, since $\rho$ has finite image, $\mathfrak{g}/\epsilon(H^0(L))$ vanishes. Thus we recover completely the result of Kapovich-Millson: $\Ohat_\rho$ has a presentation with generators of weight $1,2$ and relations of weight $2,3,4$.

\item In the above case, assume that $H^1(Y)$ is pure of weight $2$. Then $H^1(L)$ is also pure of weight $2$. And $H^2(L)$ is again limited to weights $2,3,4$. Thus the only possible non-zero operation is $\ell_2(x,x)$ for $x$ of weight $2$. So we recover the main result of \cite{Lefevre1}: in this case $\Ohat_\rho$ is quadratic. And we recover some form of \emph{purity implies formality} (see \cite{CiriciHorel}): the purity of weights implies some partial formality of $L$ hence it behaves as in the compact case.

\item If $X\subset\overline{X}$ (with $\overline{X}$ compact Kähler) and $\rho$ is the monodromy of a polarized VHS over $X$ extendable to $\overline{X}$, then $H^1(L)$ has again weights $1,2$ and $H^2(L)$ has weights $2,3,4$. Thus we recover the same result as with finite images except that in this case we also have a non-zero part $\mathfrak{g}/\epsilon(H^0(L))$.

\item If $X$ is a curve and $\rho$ comes from a polarizable VHS with unipotent monodromy, then the lowest possible weight of $H^1(L)$ is $1$ and the lowest possible weight of $H^2(L)$ is again $2$. The higher weights are not so easy to describe, depending on the behavior of the VHS near the singularities, and are not bounded without further hypothesis on the monodromy. This is however enough to conclude that there is a finite weighted-homogeneous presentation for $\Ohat_\rho$.

\item In the most general case of $X\subset \overline{X}$ (with $\overline{X}$ compact Kähler) and $\rho$ coming from a pure polarizable VHS over $X$ with unipontent monodromy then again we know that $H^1(L)$ has lowest possible weight~$1$ and we conclude as for curves, and again we cannot bound the weights in this generality.
\end{enumerate}

If $\mathbb{V}$ is mixed then a priori we cannot ensure that $H^1(L)$ has weights greater or equal to zero without further hypotheses, and we cannot prove this way that there are only finitely many equations, but our method still allows us to write down weighted-homogeneous equations for $\Ohat_\rho$.

\bibliographystyle{amsalpha}
\bibliography{Bibliographie}

\vspace{2cm}

\end{document}